\documentclass[11pt,fullpage, reqno]{amsart}
\usepackage{amsmath,amssymb,amsthm,mathtools,mathrsfs}
\usepackage{booktabs}
\usepackage[colorlinks=true,linkcolor=blue,citecolor=blue,urlcolor=blue]{hyperref}
\hypersetup{pdftitle={Weighted, GIT and Neron-Tate heights on genus-two moduli and Jacobians},pdfauthor={Sajad Salami}}
\allowdisplaybreaks
\numberwithin{equation}{section}

\usepackage{fullpage}
\newtheorem{theorem}{Theorem}[section]
\newtheorem{proposition}[theorem]{Proposition}
\newtheorem{corollary}[theorem]{Corollary}
\newtheorem{lemma}[theorem]{Lemma}
\newtheorem{maintheorem}{Theorem}

\newtheorem{maincorollary}[maintheorem]{Corollary}
\theoremstyle{definition}
\newtheorem{definition}[theorem]{Definition}
\newtheorem{example}[theorem]{Example}
\theoremstyle{remark}
\newtheorem{remark}[theorem]{Remark}

\DeclareMathOperator{\Proj}{Proj}
\DeclareMathOperator{\Pic}{Pic}
\DeclareMathOperator{\Disc}{Disc}
\DeclareMathOperator{\Res}{Res}
\DeclareMathOperator{\Sym}{Sym}
\DeclareMathOperator{\Gal}{Gal}
\DeclareMathOperator{\SL}{SL}
\DeclareMathOperator{\GL}{GL}
\DeclareMathOperator{\Jac}{Jac}
\DeclareMathOperator{\divisor}{div}
\newcommand{\Q}{\mathbf Q}
\newcommand{\Z}{\mathbf Z}
\newcommand{\R}{\mathbf R}

\newcommand{\F}{\mathbf F}
\newcommand{\PP}{\mathbf P}
\newcommand{\Qbar}{\overline{\mathbf Q}}
\newcommand{\Kbar}{\overline K}
\newcommand{\OO}{\mathcal O}
\newcommand{\ww}{\boldsymbol w}
\newcommand{\hhat}{\widehat h}
\newcommand{\tors}{\mathrm{tors}}
\newcommand{\inv}{\mathrm{inv}}
\newcommand{\GIT}{\mathrm{GIT}}
\newcommand{\Xq}{\mathcal X}
\newcommand{\Hq}{\mathcal H}
\newcommand{\Jfam}{\mathscr J}

\title[Weighted, GIT and N\'eron--Tate heights in genus two]{Weighted, GIT and N\'eron--Tate heights\\ on genus-two moduli and Jacobians}
\author{Sajad Salami}
\address{Institute of Mathematics, Statistics and Computing, Rio de Janeiro State University (UERJ), Rua S\~ao Francisco Xavier 524, Maracan\~a, Rio de Janeiro, Brazil}
\date{}
\subjclass[2020]{Primary 11G50, 11G30; Secondary 11G10, 14G40, 14H40, 14L24}
\keywords{Weighted heights, GIT heights, N\'eron--Tate heights, genus-two curves, binary sextics, Igusa--Clebsch invariants, specialization of heights}

\begin{document}

\begin{abstract}
We compare three heights attached to a curve of genus two: the weighted height of its Igusa--Clebsch invariants, the GIT height of a binary sextic defining it, and the N\'eron--Tate height on its Jacobian. For a morphism from an abelian variety to a weighted projective space we define a weighted N\'eron--Tate height and show that the weighted height differs from it by a linear form and a bounded function; the linear form vanishes when inversion is induced by an automorphism of the weighted projective space. For binary sextics we prove the explicit inequality $h_{\inv}\le h_{\text{GIT}} + \frac{1}{2} \log 302$, with nonnegative local differences at the finite places that vanish at places of good reduction. We then construct a family of semistable binary sextics over $E\times E$, where $E:y^2=x^3-2$, whose relative Jacobian carries a section of infinite order, and we show that the Jacobian over the base point is absolutely simple. Along every curve in the parameter space through the base point, the N\'eron--Tate height of the section is a positive multiple of the weighted moduli height up to an error of square-root size, and up to a bounded error when the curve is rational. As applications we obtain explicit families of genus-two Jacobians of positive rank over $\mathbf{Q}$, an infinite congruence lattice of fibres of positive height, and canonical-limit formulas along elliptic curves in $E\times E$.
\end{abstract}

\maketitle

\section{Introduction}\label{sec:intro}

Let $C$ be a smooth projective curve of genus two over a number field. Three heights are naturally attached to $C$. The first is the \emph{weighted height} of its Igusa--Clebsch invariants $(I_2:I_4:I_6:I_{10})$, regarded as a point of the weighted projective space $\PP(2,4,6,10)$. The second is a \emph{GIT height} of a binary sextic $F$ with $C:y^2=F(x,1)$, obtained by minimizing a norm of $F$ over its $\SL_2$-orbit at every place. The third is the \emph{N\'eron--Tate height} on the Jacobian $J=\Jac(C)$. Weighted heights on weighted projective spaces were introduced in \cite{BGS} and studied for weighted varieties in \cite{SalamiShaska}. GIT heights go back to the work of Burnol \cite{Burnol}, Cornalba and Harris \cite{CornalbaHarris} and Zhang \cite{Zhang1994,Zhang1996}; see also \cite{RabindranathSawin}. Curri \cite{Curri} observed similarities between the weighted height and the GIT height of binary forms, and the relation between the two was studied in \cite{ShaskaShaska}. N\'eron--Tate heights on genus-two Jacobians are computed algorithmically in \cite{MS}. The purpose of this paper is to make the relations among these three heights precise.

The three heights do not live on the same space. The first two are functions on the moduli space, while the third is a function on the Jacobian, so any comparison has to pass through a family: a parameter variety $S$, a morphism from $S$ to the moduli space, and a section of the relative Jacobian over $S$. Since the coarse moduli space of smooth genus-two curves is affine, a nonconstant family of this kind over a complete base must have singular fibres (Proposition~\ref{prop:affine}). This forces us to work on the GIT compactification rather than on the moduli space itself.

We take the parameter space to be an abelian surface. On an abelian variety every Weil height is, up to a bounded function, a quadratic form plus a linear form, so the weighted height of the moduli point becomes a concrete arithmetic function of the parameter. The group law also supplies natural sets of parameters, such as multiples of a rational point, lattices and one-parameter subgroups, along which heights can be counted and compared. The families constructed here are parametrized by $E\times E$, where
\[
E:\ y^2=x^3-2,\qquad P_0=(3,5)\in E(\Q),
\]
and they factor through a finite morphism $E\times E\to\PP^2$ of degree $18$. Throughout, $\hhat_E$ is the N\'eron--Tate height attached to $\OO_E(O)$, normalized by
\[
\hhat_E(P)=\tfrac12\lim_{j\to\infty}4^{-j}h\bigl(x([2^j]P)\bigr),
\]
and $\langle\ ,\ \rangle_E$ is the associated pairing. One computes $\hhat_E(P_0)=0.6748\ldots$ (Example~\ref{ex:two}).

We now describe the results, in the order in which they are proved.

\subsection*{Weighted N\'eron--Tate heights}
The first result is a canonical form for weighted heights on abelian varieties. Let $A$ be an abelian variety over a number field $K$ and $f\colon A\to\PP(\ww)$ a morphism to a weighted projective space with weighted height $h_{\ww}$ (see \eqref{eq:hw}).

\begin{maintheorem}\label{thm:A}
\begin{enumerate}
\item For every $P\in A(\Kbar)$ the limit
\[
\hhat_{f,\ww}(P)=\lim_{j\to\infty}4^{-j}h_{\ww}\bigl(f([2^j]P)\bigr)
\]
exists. The function $\hhat_{f,\ww}$ is a quadratic form, and there is a homomorphism $\ell_f\colon A(\Kbar)\to\R$ such that
\[
h_{\ww}(f(P))=\hhat_{f,\ww}(P)+\ell_f(P)+O(1)\qquad(P\in A(\Kbar)).
\]
\item If $f$ is finite, then $\hhat_{f,\ww}$ is nonnegative and vanishes exactly on torsion points, and the number of points of a finitely generated subgroup $\Gamma\subseteq A(K)$ of rank $r\ge1$ with $h_{\ww}(f(P))\le X$ is $c_\Gamma X^{r/2}+O(X^{(r-1)/2})$ with an explicit $c_\Gamma>0$.
\item Suppose that $f\circ[-1]=\bar\varphi\circ f$, where $\bar\varphi$ is the automorphism of $\PP(\ww)$ induced by a graded automorphism $\varphi$ of the weighted polynomial ring. Then $\ell_f=0$ and $h_{\ww}\circ f=\hhat_{f,\ww}+O(1)$.
\end{enumerate}
\end{maintheorem}

Theorem~\ref{thm:A} is proved in Section~\ref{sec:abelian}. The linear form $\ell_f$ does not vanish in general (Example~\ref{ex:linear}), so part (3) is the natural condition under which the weighted height is a N\'eron--Tate height up to a bounded function. For a Weierstrass cubic in $\PP(1,2,3)$ part (3) applies, and the weighted N\'eron--Tate height is the usual one (Corollary~\ref{cor:ellcanon}).

\subsection*{Weighted heights and GIT heights}
The second result compares two heights on the moduli space itself. For a binary sextic $F$ over a number field $K$ and a place $v$ put $\rho_v(F)=\inf_{g\in\SL_2(\Kbar_v)}\|gF\|_v$, where $\|\cdot\|_v$ is the maximum of the absolute values of the coefficients, and let $h_{\GIT}(F)=[K:\Q]^{-1}\sum_v n_v\log\rho_v(F)$. Let $h_{\inv}(F)$ be the weighted height of $(I_2:I_4:I_6:I_{10})(F)$ with the Igusa--Clebsch normalization \eqref{eq:IC}; its local terms are $\log W_v(F)$, where $W_v(F)=\max_d|I_d(F)|_v^{1/d}$.

\begin{maintheorem}\label{thm:B}
Let $F$ be a GIT-semistable binary sextic over $\Qbar$, that is, a nonzero sextic without roots of multiplicity $\ge4$.
\begin{enumerate}
\item $h_{\inv}(F)\le h_{\GIT}(F)+\tfrac12\log302$.
\item At every finite place the local contribution $\log\rho_v(F)-\log W_v(F)$ to $h_{\GIT}-h_{\inv}$ is nonnegative, and it vanishes if $F$ has good reduction at $v$.
\item $h_{\GIT}(F)\le h(gF)$ for every $g\in\GL_2(\Qbar)$. In particular, if $C:y^2=F(x,1)$ is a curve of genus two, then $h_{\inv}([C])\le h(gF)+\tfrac12\log302$ for every model $gF$ of $C$.
\end{enumerate}
\end{maintheorem}

The constant comes from the coefficient sums $302$, $4192$, $283842$, $1561993$ of $I_2,I_4,I_6,I_{10}$. Theorem~\ref{thm:B} is a special case of a general inequality for reductive groups (Theorem~\ref{thm:gitgeneral}) and is proved in Section~\ref{sec:sextics}. The reverse inequality holds with a non-explicit constant by the work of Burnol and Zhang (Remark~\ref{rem:reverse}).

\subsection*{A family with a section of infinite order}
To bring the Jacobian into the picture we need an explicit family together with a section of its relative Jacobian that is not torsion. Let $a,b,c$ be the sections \eqref{eq:abc} of $L=p_1^*\OO_E(3O)\otimes p_2^*\OO_E(3O)$ on $E\times E$ and let
\[
f_0(x)=x(x-1)(x-2)(x-3)(x+20)(x+5).
\]
For $R\in E\times E$ consider the binary sextic
\begin{multline*}
F_R=a^2Z^6f_0(X/Z)+abXZ^5+acX^2Z^4+b^2(X^5Z+XZ^5)\\
+bc(X^4Z^2+X^2Z^4)+c^2(X^4Z^2+XZ^5),
\end{multline*}
the curve $C_R:y^2=F_R(x,1)$, its Weierstrass point $W=(0,0)$ and the point $P_\infty^+$ at infinity with $y/x^3\to a$.

\begin{maintheorem}\label{thm:C}
\begin{enumerate}
\item Every $F_R$ is GIT-semistable. The moduli map $\Phi\colon E\times E\to\PP(2,4,6,10)$ is finite onto a surface, $\Phi^*\OO(60)\simeq L^{\otimes120}$, and $\Phi$ factors through the finite morphism $[a:b:c]\colon E\times E\to\PP^2$ of degree $18$.
\item The section $\sigma=[P_\infty^+-W]$ of the relative Jacobian has infinite order. For every $(n,m)\in\Z^2$ the fibre over $R_{n,m}=([342n]P_0,[342m]P_0)$ is smooth, $\sigma(R_{n,m})$ has positive N\'eron--Tate height, and
\[
\#\{(n,m)\in\Z^2:\ h_{\inv}(\Phi(R_{n,m}))\le X\}=\frac{\pi}{6\cdot342^2\,\hhat_E(P_0)}X+O(\sqrt X).
\]
\item For each parameter $R$ listed in Proposition~\ref{prop:small}, for instance $(P_0,O)$, $(P_0,P_0)$ and $([2]P_0,O)$, the group $\Jac(C_R)(\Q)$ is torsion-free and contains the point $\sigma(R)$ of infinite order; in particular it has positive rank.
\end{enumerate}
\end{maintheorem}

Theorem~\ref{thm:C} is proved in Sections~\ref{sec:surface} and~\ref{sec:sections}. Part (2) rests on point counts at $13$ and $17$ and needs no information about the rational torsion of the varying Jacobians.

\subsection*{N\'eron--Tate heights versus moduli heights}
With the section in hand, we can compare the N\'eron--Tate height of $\sigma$ with the weighted height of the moduli point. The comparison takes place along curves in the parameter plane $\PP^2$ of the family. The reference curve $C_0:y^2=f_0(x)$ sits over the base point $e=(1:0:0)$, and every curve we consider passes through $e$.

\begin{maintheorem}\label{thm:D}
Let $B$ be a smooth projective geometrically connected curve over a number field and $\beta\colon B\to\PP^2$ a nonconstant morphism whose image contains $e$. Let $d_\beta=\deg\beta^*\OO(1)$. Then the Jacobian family over $B$ has trivial trace, and there is $\delta_\beta>0$ such that, with $s_\beta=\delta_\beta/(2d_\beta)$,
\[
\hhat_{J_b}(\sigma(\beta(b)))=s_\beta\,h_{\inv}(\beta(b))+O\bigl(\sqrt{h_{\inv}(\beta(b))+1}\bigr)
\]
for all $b\in B(\Qbar)$ whose fibre is smooth with $a\ne0$, where $h_{\inv}(\beta(b))$ is the weighted moduli height of that fibre. If $B$ has genus zero, the error term is $O(1)$. The slope $s_\beta$ depends only on the image curve $\beta(B)$. For each $n\ge1$ only finitely many $b\in B(\Qbar)$ of degree at most $n$ have $\sigma(\beta(b))$ torsion.
\end{maintheorem}

Theorem~\ref{thm:D} is proved in Section~\ref{sec:specialization}. Its proof uses the specialization theorem of Tate, Silverman and Carney, and the slope $s_\beta$ is defined through a function-field N\'eron--Tate height; its numerical value is a natural target for further work (Section~\ref{sec:questions}). The case of the line $c=0$ gives a completely explicit family over $\Q$.

\begin{maincorollary}\label{cor:E}
For $u\in\Q$ let $C_u:\ y^2=f_0(x)+ux+u^2(x^5+x)$. There are $\delta>0$ and a finite set $\Sigma\subset\Q$ such that for every $u\in\Q\setminus\Sigma$ the curve $C_u$ has genus two, the class $[P_\infty^+-W]$ has infinite order in $\Jac(C_u)(\Q)$, and
\[
\hhat_{\Jac(C_u)}([P_\infty^+-W])=\delta\,h(u)+O(1)=\tfrac{\delta}{2}\,h_{\inv}([C_u])+O(1).
\]
The set $\Sigma$ contains none of $0,\pm1,\pm2,\pm3,\pm\tfrac13,\tfrac12,\tfrac23,\tfrac32$. The number of $\Qbar$-isomorphism classes of curves $C_u$ with $u\notin\Sigma$ and $\exp h_{\inv}([C_u])\le Y$ is $\gg Y$.
\end{maincorollary}

For the elliptic curves $T\mapsto([u]T,[v]T)$ in $E\times E$ we obtain quadratic--linear expansions $\hhat_J(\sigma)=\delta_{u,v}\hhat_E(T)+2\langle T,R_{u,v}\rangle_E+O(1)$, the symmetries $\delta_{du,dv}=d^2\delta_{u,v}$ and $\delta_{v,u}=\delta_{u,v}$, and exact canonical limits that identify the slope with a ratio of a function-field N\'eron--Tate height and a weighted N\'eron--Tate height (Theorems~\ref{thm:ray} and~\ref{thm:canonical-identities}).

\subsection*{Strategy}
The proofs combine three ingredients. The comparison between weighted heights and canonical heights on the parameter space is a formal consequence of the height machine once the pullback of the quotient polarization has been identified, and we identify it explicitly. The comparison with GIT heights rests on an elementary inequality between invariants and orbit-minimal norms. The comparison with the N\'eron--Tate height of the section needs two arithmetic inputs, both certified by finite computations: the section has infinite order, which we prove by point counts modulo $13$ and $17$, and the Jacobian over $e$ is absolutely simple, which we prove by reduction modulo $31$. Once these are known, the specialization theorem turns the comparison into a comparison of two line bundles on the base curve.

\subsection*{Organization}
Section~\ref{sec:weighted} recalls weighted heights. Section~\ref{sec:abelian} proves Theorem~\ref{thm:A}, and Section~\ref{sec:elliptic} treats weighted Weierstrass models. Section~\ref{sec:git} compares weighted and GIT heights for an arbitrary reductive group, and Section~\ref{sec:sextics} specializes to binary sextics and proves Theorem~\ref{thm:B}. Section~\ref{sec:pencil} studies a pencil with five fixed branch points and an exact place-by-place formula. Section~\ref{sec:surface} constructs the families over $E\times E$, and Section~\ref{sec:sections} proves Theorem~\ref{thm:C}. Section~\ref{sec:specialization} proves Theorem~\ref{thm:D}, Corollary~\ref{cor:E} and the results on elliptic curves in $E\times E$. Section~\ref{sec:questions} lists some questions, and Appendix~\ref{app:data} collects the computational certificates; a script reproducing all of them is provided as supplementary material.

\subsection*{Conventions}
$K$ denotes a number field, $M_K$ its set of places, and $|\cdot|_v$ the absolute value extending the normalized absolute value of $\Q$; we put $n_v=[K_v:\Q_v]$. Heights are absolute and logarithmic; $h$ is the standard height on projective space. For a line bundle $N$ on a projective variety, $h_N$ is a Weil height for $N$, determined up to $O(1)$. On a genus-two Jacobian with a rational Weierstrass point $W$ we use the N\'eron--Tate height attached to the symmetric theta divisor $\Theta_W$, the image of $P\mapsto[P-W]$.

\section{Weighted heights and the complete Veronese system}\label{sec:weighted}

Weighted heights are best understood through the Veronese map given by all monomials of a fixed weighted degree. This section records the exact identity relating the two heights, which is the basic tool for everything that follows.

Let $\ww=(w_0,\ldots,w_N)$ be a tuple of positive integers, and let $\PP(\ww)$ be the weighted projective space $\Proj K[X_0,\ldots,X_N]$ with $\deg X_i=w_i$. For a nonzero vector $x=(x_0,\ldots,x_N)\in K^{N+1}$ set
\begin{equation}\label{eq:hw}
h_{\ww}([x_0:\cdots:x_N])=\frac1{[K:\Q]}\sum_{v\in M_K}n_v\log\max_{0\le i\le N}|x_i|_v^{1/w_i}.
\end{equation}
By the product formula this is unchanged under $x_i\mapsto\lambda^{w_i}x_i$, and passing to a field containing the coordinates defines an absolute height on $\PP(\ww)(\Kbar)$.

Fix a positive integer $m$ divisible by every $w_i$ and let $\mathcal I_m=\{\alpha\in\Z_{\ge0}^{N+1}:\sum_iw_i\alpha_i=m\}$. The monomials of weighted degree $m$ define a morphism
\[
\nu_m\colon\PP(\ww)\longrightarrow\PP^{|\mathcal I_m|-1},\qquad x\longmapsto(x^\alpha)_{\alpha\in\mathcal I_m},
\]
without base points, since the pure powers $X_i^{m/w_i}$ are among these monomials.

\begin{proposition}\label{prop:monomial}
For every $x\in\PP(\ww)(\Kbar)$ one has $h(\nu_m(x))=m\,h_{\ww}(x)$, and $\nu_m^*\OO(1)\simeq\OO_{\PP(\ww)}(m)$.
\end{proposition}

\begin{proof}
Put $R_v=\max_i|x_i|_v^{1/w_i}$. For $\alpha\in\mathcal I_m$ we have $|x^\alpha|_v\le\prod_iR_v^{w_i\alpha_i}=R_v^m$, with equality for the pure power $x_j^{m/w_j}$ of an index $j$ realizing the maximum. Hence $\max_\alpha|x^\alpha|_v=R_v^m$ at every place, and summing gives the identity of heights. The second statement holds because $\nu_m$ is given by generating sections of the invertible sheaf $\OO(m)$.
\end{proof}

The proof uses only the pure powers, but the full list of monomials is needed to obtain a morphism that separates points. The pure powers alone give the same maximum but need not separate points: in $\PP(1,1,2)$ the points $(1:1:1)$ and $(1:-1:1)$ have the same squares. For heights only the presence of all pure powers matters. For general weights $\OO(1)$ need not be invertible, which is why we work with $\OO(m)$.

\begin{corollary}\label{cor:pullback}
Let $X$ be a projective variety over $K$, $f\colon X\to\PP(\ww)$ a morphism and $N=f^*\OO_{\PP(\ww)}(m)$. Then $m\,h_{\ww}(f(P))=h_N(P)+O(1)$ uniformly for $P\in X(\Kbar)$.
\end{corollary}

\begin{proof}
By Proposition~\ref{prop:monomial} the left side equals $h(\nu_m(f(P)))$, and $(\nu_m\circ f)^*\OO(1)\simeq N$.
\end{proof}

\section{Weighted heights on abelian varieties}\label{sec:abelian}

We now restrict to morphisms whose source is an abelian variety. Here the height machine is especially rigid: every Weil height is a quadratic form plus a linear form plus a bounded function. This section makes that decomposition explicit for weighted heights, isolates its canonical part, which we call the weighted N\'eron--Tate height, and proves Theorem~\ref{thm:A}.

Let $A/K$ be an abelian variety. We write the group law of $\Pic(A)\otimes\Q$ additively and, for $N\in\Pic(A)$, put $N^{\pm}=\tfrac12(N\pm[-1]^*N)$. The class $N^+$ is symmetric, and $N^-$ lies in $\Pic^0(A)\otimes\Q$ because $[-1]$ acts trivially on the N\'eron--Severi group.

\subsection{The quadratic--linear decomposition}

We begin with the two building blocks of the height machine on $A$; heights attached to classes in $\Pic(A)\otimes\Q$ are defined by linearity.

\begin{lemma}\label{lem:eigen}
Let $S\in\Pic(A)\otimes\Q$ be symmetric and $M\in\Pic^0(A)\otimes\Q$. There are functions $\hhat_S$ and $\lambda_M$ on $A(\Kbar)$, with $\hhat_S$ a quadratic form and $\lambda_M$ a homomorphism, such that $h_S=\hhat_S+O(1)$ and $h_M=\lambda_M+O(1)$. Thus $\hhat_S([n]P)=n^2\hhat_S(P)$ and $\lambda_M([n]P)=n\lambda_M(P)$ for all $n\in\Z$.
\end{lemma}

\begin{proof}
By the theorem of the cube, $[2]^*S\simeq S^{\otimes4}$ and $[2]^*M\simeq M^{\otimes2}$ \cite[Sections~6 and~8]{MumfordAV}. Hence $h_S([2]P)=4h_S(P)+O(1)$ and $h_M([2]P)=2h_M(P)+O(1)$, and Tate's telescoping argument shows that $\hhat_S=\lim4^{-j}h_S\circ[2^j]$ and $\lambda_M=\lim2^{-j}h_M\circ[2^j]$ exist and differ from $h_S$, $h_M$ by bounded functions. The theorem of the cube gives the parallelogram law for $\hhat_S$. For $M$, the addition map $\mu$ satisfies $\mu^*M\simeq p_1^*M\otimes p_2^*M$, so $h_M(P+Q)=h_M(P)+h_M(Q)+O(1)$; applying this to $([2^j]P,[2^j]Q)$ and dividing by $2^j$ shows that $\lambda_M$ is additive \cite{BombieriGubler,HindrySilverman}.
\end{proof}

\begin{theorem}\label{thm:decomposition}
Let $f\colon A\to\PP(\ww)$ be a morphism, $m$ a common multiple of the weights, $N=f^*\OO_{\PP(\ww)}(m)$, $q_N=\hhat_{N^+}$ and $\ell_N=\lambda_{N^-}$. Then
\begin{equation}\label{eq:decomp}
m\,h_{\ww}(f(P))=q_N(P)+\ell_N(P)+O(1)
\end{equation}
uniformly on $A(\Kbar)$. Consequently $m\,h_{\ww}(f([n]P))=n^2q_N(P)+n\ell_N(P)+O(1)$ uniformly in $P$ and $n\in\Z$, and
\begin{equation}\label{eq:evenodd}
\begin{aligned}
m\bigl(h_{\ww}(f(P))+h_{\ww}(f(-P))\bigr)&=2q_N(P)+O(1),\\
m\bigl(h_{\ww}(f(P))-h_{\ww}(f(-P))\bigr)&=2\ell_N(P)+O(1).
\end{aligned}
\end{equation}
\end{theorem}

\begin{proof}
By Corollary~\ref{cor:pullback} the left side of \eqref{eq:decomp} is $h_N+O(1)$, and $h_N=h_{N^+}+h_{N^-}+O(1)$ since $N=N^++N^-$. Lemma~\ref{lem:eigen} gives \eqref{eq:decomp}. Applying it at $[n]P$ and using the scaling properties gives the second statement, and adding or subtracting the formulas for $P$ and $-P$ gives \eqref{eq:evenodd}.
\end{proof}

\begin{example}\label{ex:linear}
The linear term is genuinely present in general. Let $E/K$ be an elliptic curve, $R\in E(K)$ a point of infinite order and $f_R\colon E\to\PP^1$, $Q\mapsto x(Q-R)$. Then $f_R^*\OO(1)\simeq\OO_E(2(R))$, and the canonical height of the $x$-coordinate gives
\[
h(f_R(Q))=2\hhat_E(Q-R)+O(1)=2\hhat_E(Q)-4\langle Q,R\rangle_E+O_R(1).
\]
For $Q=[n]R$ this is $2n^2\hhat_E(R)-4n\hhat_E(R)+O(1)$, which differs from the quadratic term by an unbounded amount. The bundle $f_R^*\OO(1)$ is ample but not symmetric.
\end{example}

\subsection{The weighted N\'eron--Tate height}
Theorem~\ref{thm:decomposition} suggests normalizing away both the bounded and the linear part. Doubling does this, because the quadratic part grows by a factor of four while the linear part grows only by a factor of two.

\begin{definition}\label{def:wnt}
Let $f\colon A\to\PP(\ww)$ be a morphism. The \emph{weighted N\'eron--Tate height} of $f$ is
\[
\hhat_{f,\ww}(P)=\lim_{j\to\infty}4^{-j}\,h_{\ww}\bigl(f([2^j]P)\bigr),\qquad P\in A(\Kbar).
\]
\end{definition}

\begin{theorem}\label{thm:wnt}
In the notation of Theorem~\ref{thm:decomposition}, the limit defining $\hhat_{f,\ww}$ exists, does not depend on $m$, and
\begin{equation}\label{eq:wnt}
\hhat_{f,\ww}=\frac1m\hhat_{N^+},\qquad h_{\ww}\circ f=\hhat_{f,\ww}+\frac1m\lambda_{N^-}+O(1).
\end{equation}
It is a quadratic form, $\hhat_{f,\ww}([n]P)=n^2\hhat_{f,\ww}(P)$, and
\[
\hhat_{f,\ww}(P+Q)+\hhat_{f,\ww}(P-Q)=2\hhat_{f,\ww}(P)+2\hhat_{f,\ww}(Q).
\]
If $f$ is finite, then $\hhat_{f,\ww}(P)\ge0$ with equality if and only if $P$ is a torsion point.
\end{theorem}

\begin{proof}
Apply Theorem~\ref{thm:decomposition} with $n=2^j$ and divide by $m4^j$: the linear term contributes $2^{-j}\ell_N(P)/m$ and the error is $O(4^{-j})$. This proves \eqref{eq:wnt}; replacing $m$ by $km$ replaces $N$ by $N^{\otimes k}$ and leaves $\frac1m\hhat_{N^+}$ unchanged. If $f$ is finite, $N$ is ample because $\OO(m)$ is ample on $\PP(\ww)$, and then $N^+=\frac12(N+[-1]^*N)$ is ample; the last assertion is the classical property of N\'eron--Tate heights of ample symmetric classes.
\end{proof}

In general the linear term $\frac1m\lambda_{N^-}$ in \eqref{eq:wnt} is present, as Example~\ref{ex:linear} shows. The next theorem gives a simple criterion, checkable on equations, for it to vanish.

\begin{theorem}\label{thm:symmetric}
Suppose that there is a graded $K$-algebra automorphism $\varphi$ of $K[X_0,\ldots,X_N]$ such that the induced automorphism $\bar\varphi$ of $\PP(\ww)$ satisfies $f\circ[-1]=\bar\varphi\circ f$. Then $\ell_N=0$ and
\[
h_{\ww}(f(P))=\hhat_{f,\ww}(P)+O(1)\qquad(P\in A(\Kbar)).
\]
If $\varphi$ is diagonal, $X_i\mapsto\varepsilon_iX_i$ with roots of unity $\varepsilon_i$, then $h_{\ww}(f(-P))=h_{\ww}(f(P))$ exactly.
\end{theorem}

\begin{proof}
Write $\varphi(X_i)=\varphi_i(X)$ with $\varphi_i$ weighted homogeneous of degree $w_i$. At every place, $|\varphi_i(x)|_v^{1/w_i}\le c_v\max_j|x_j|_v^{1/w_j}$ with $c_v=1$ for almost all $v$, and the same holds for $\varphi^{-1}$. Hence $h_{\ww}(\bar\varphi(y))=h_{\ww}(y)+O(1)$ on $\PP(\ww)(\Kbar)$, and therefore $h_{\ww}(f(-P))=h_{\ww}(f(P))+O(1)$. By \eqref{eq:evenodd}, $\ell_N(P)=O(1)$ uniformly in $P$. Since $\ell_N([n]P)=n\ell_N(P)$ is bounded in $n$, we get $\ell_N=0$. For diagonal $\varphi$ the local maxima in \eqref{eq:hw} are unchanged.
\end{proof}

\begin{corollary}\label{cor:para}
If $f$ satisfies the hypothesis of Theorem~\ref{thm:symmetric}, then
\[
\bigl|h_{\ww}(f(P+Q))+h_{\ww}(f(P-Q))-2h_{\ww}(f(P))-2h_{\ww}(f(Q))\bigr|\le C
\]
for a constant $C$ and all $P,Q\in A(\Kbar)$. For arbitrary $f$ the same holds for the even height $\frac12\bigl(h_{\ww}(f(P))+h_{\ww}(f(-P))\bigr)$.
\end{corollary}

\begin{proof}
Both functions differ by a bounded amount from the quadratic form $\hhat_{f,\ww}$, by Theorem~\ref{thm:symmetric} and \eqref{eq:evenodd} respectively.
\end{proof}

\subsection{Counting on Mordell--Weil groups}
Once the weighted N\'eron--Tate height is known to be a positive definite quadratic form, counting points of bounded weighted height on a Mordell--Weil group becomes a lattice-point problem; the linear term only shifts the centre of the ellipsoids.

\begin{theorem}\label{thm:count}
Let $f\colon A\to\PP(\ww)$ be finite, $\Gamma\subseteq A(K)$ a finitely generated subgroup of rank $r\ge1$ with torsion subgroup $T$, and $P_1,\ldots,P_r$ a basis of $\Gamma/T$. Let $G_\Gamma$ be the Gram matrix of $\hhat_{f,\ww}$ on $P_1,\ldots,P_r$ and $V_r=\pi^{r/2}/\Gamma(1+\tfrac r2)$. Then
\begin{equation}\label{eq:count}
\#\{P\in\Gamma:\ h_{\ww}(f(P))\le X\}=\frac{|T|\,V_r}{\sqrt{\det G_\Gamma}}\,X^{r/2}+O\bigl(X^{(r-1)/2}\bigr),
\end{equation}
where for $r=1$ the error term is $O(1)$.
\end{theorem}

\begin{proof}
Write $P=T_0+\sum n_iP_i$ with $T_0\in T$ and $n\in\Z^r$. By Theorem~\ref{thm:wnt}, $h_{\ww}(f(P))=n^tG_\Gamma n+b^tn+O(1)$ for some $b\in\R^r$, uniformly in $P$; both the quadratic and the linear part vanish on torsion. Completing the square, the condition $h_{\ww}(f(P))\le X$ is sandwiched between $(n+c)^tG_\Gamma(n+c)\le X\mp C_1$ with $c=\frac12G_\Gamma^{-1}b$. The matrix $G_\Gamma$ is positive definite by Theorem~\ref{thm:wnt}. A translated ellipsoid of radius $\sqrt X$ contains $V_rX^{r/2}/\sqrt{\det G_\Gamma}+O(X^{(r-1)/2})$ lattice points, and multiplying by $|T|$ gives \eqref{eq:count}.
\end{proof}

Theorem~\ref{thm:A} is now proved: part (1) is Theorems~\ref{thm:decomposition} and~\ref{thm:wnt}, part (2) is Theorems~\ref{thm:wnt} and~\ref{thm:count}, and part (3) is Theorem~\ref{thm:symmetric}. In particular $\{P\in\Gamma:h_{\ww}(f(P))\le X\}$ is finite and grows like $X^{r/2}$. Finiteness of $f$ cannot be dropped: a constant map, or one factoring through a quotient of smaller dimension, has a degenerate pullback height.

\section{Weighted Weierstrass models}\label{sec:elliptic}

The simplest weighted embedding of an abelian variety is a Weierstrass cubic in $\PP(1,2,3)$. We treat it separately for two reasons: the comparison constants can be made completely explicit, and the curve $E:y^2=x^3-2$ that appears here is the building block of the parameter spaces of Sections~\ref{sec:surface}--\ref{sec:specialization}.

Let $E:y^2=x^3+a_4x+a_6$ with $a_4,a_6\in\Z$ and $4a_4^3+27a_6^2\ne0$, embedded in $\PP(1,2,3)$ by $y^2=x^3+a_4xz^4+a_6z^6$. For an affine point $P=(x,y)\in E(\Q)$,
\[
h_{1,2,3}(P)=\sum_{v\in M_\Q}\log\max\{1,|x|_v^{1/2},|y|_v^{1/3}\},
\]
and $h_{1,2,3}(O)=0$, where $O=(0:1:1)$.

\begin{proposition}\label{prop:local}
Let $P=(x,y)\in E(\Q)$ be affine. At every prime $p$, $\max\{1,|x|_p^{1/2},|y|_p^{1/3}\}=\max\{1,|x|_p^{1/2}\}$, and at the real place
\[
\max\{1,|x|^{1/2}\}\le\max\{1,|x|^{1/2},|y|^{1/3}\}\le(1+|a_4|+|a_6|)^{1/6}\max\{1,|x|^{1/2}\}.
\]
Consequently
\begin{equation}\label{eq:effective}
\tfrac12h(x(P))\le h_{1,2,3}(P)\le\tfrac12h(x(P))+\tfrac16\log(1+|a_4|+|a_6|).
\end{equation}
\end{proposition}

\begin{proof}
If $|x|_p\le1$ then $|y|_p^2=|x^3+a_4x+a_6|_p\le1$. If $|x|_p>1$, the term $x^3$ strictly dominates and $|y|_p^{1/3}=|x|_p^{1/2}$. At the real place, with $M_\infty=\max\{1,|x|\}$, we have $|y|^2\le(1+|a_4|+|a_6|)M_\infty^3$.
\end{proof}

Proposition~\ref{prop:local} shows that the $y$-coordinate contributes at most a bounded amount, so the weighted height is controlled by the $x$-coordinate alone. Its canonical form is the usual N\'eron--Tate height.

\begin{corollary}\label{cor:ellcanon}
The weighted N\'eron--Tate height of the embedding $E\hookrightarrow\PP(1,2,3)$ is $\hhat_E$, and
\[
h_{1,2,3}([n]P)=n^2\hhat_E(P)+O_E(1)
\]
uniformly in $P\in E(\Q)$ and $n\in\Z$.
\end{corollary}

\begin{proof}
Inversion is $(z:x:y)\mapsto(z:x:-y)$, a diagonal graded automorphism, so Theorem~\ref{thm:symmetric} applies. By \eqref{eq:effective} and the canonical height of the $x$-coordinate, $h_{1,2,3}=\frac12h\circ x+O(1)=\hhat_E+O(1)$.
\end{proof}

\begin{example}\label{ex:two}
Let $E:y^2=x^3-2$ and $P_0=(3,5)$. Then $h_{1,2,3}(P_0)=\frac12\log3$, $[2]P_0=(\frac{129}{100},-\frac{383}{1000})$ and $h_{1,2,3}([2]P_0)=\frac12\log129$. The group $E(\Q)_{\tors}$ is trivial: a nonzero torsion point would be integral with $y=0$ or $y^2\mid108$ by the Nagell--Lutz theorem \cite{SilvermanAEC}, but $2$ is not a cube and none of $3,6,11,38$ is a cube. In particular $P_0$ has infinite order. The convergence in Corollary~\ref{cor:ellcanon} is visible numerically: with $t_j=4^{-j}h_{1,2,3}([2^j]P_0)$ one finds
\[
\begin{array}{c|ccccccc}
j&0&1&2&3&4&6&8\\\hline
t_j&0.5493&0.6075&0.6742&0.6742&0.6742&0.6748&0.6748
\end{array}
\]
so that $\hhat_E(P_0)=0.6748\ldots$.
\end{example}

\begin{corollary}\label{cor:ellcount}
If $E(\Q)$ has rank $r\ge1$ with N\'eron--Tate regulator $\mathrm{Reg}_E$ (for $\hhat_E$), then
\[
\#\{P\in E(\Q):\ h_{1,2,3}(P)\le X\}=\frac{|E(\Q)_{\tors}|\,V_r}{\sqrt{\mathrm{Reg}_E}}X^{r/2}+O_E\bigl(X^{(r-1)/2}\bigr).
\]
\end{corollary}

\begin{proof}
This is Theorem~\ref{thm:count} for the finite map $E\hookrightarrow\PP(1,2,3)$.
\end{proof}

The quotient of $E$ by $[-1]$ gives a first, very small example of the interplay between weighted heights and GIT quotients studied in the next section.

\begin{example}\label{ex:gitinvolution}
The invariant ring of $K[z,x,y]/(y^2-x^3-a_4xz^4-a_6z^6)$ under $y\mapsto-y$ is $K[z,x]$, so $E/\{\pm1\}\simeq\PP(1,2)$ with quotient map $\pi(z:x:y)=(z:x)$. At affine points $h_{(1,2)}(\pi(P))=\frac12h(x(P))$ exactly, $\pi^*\OO(2)\simeq\OO_E(2O)$, and the weighted N\'eron--Tate height of $\pi$ is again $\hhat_E$. This is the simplest instance of a GIT quotient, for the finite group $\{\pm1\}$, in which the weighted height on the quotient pulls back to a N\'eron--Tate height.
\end{example}

\section{Weighted heights and GIT heights}\label{sec:git}

A quotient of a projective space by a reductive group embeds in a weighted projective space through a set of generating invariants, so the invariant height of a point of the quotient is a weighted height in disguise. A second and more intrinsic height is obtained by minimizing a norm over the orbit at every place. This section compares the two, first at the level of the polarization and then at the level of local terms.

\subsection{Invariant heights on quotients}
Let $G$ be a reductive group over $K$ acting linearly on a finite-dimensional vector space $V$, and let
\[
R=\bigoplus_{d\ge0}H^0\bigl(\PP(V),\OO(d)\bigr)^G,\qquad Y=\Proj R
\]
\cite{MumfordGIT}. Suppose $R$ is generated by homogeneous invariants $I_0,\ldots,I_r$ of degrees $w_0,\ldots,w_r$. They define a closed immersion $Y\hookrightarrow\PP(\ww)$, and for a common multiple $m$ of the $w_i$ the sheaf $\Hq=\OO_Y(m)$ is invertible and ample. The \emph{invariant height} of $y\in Y(\Kbar)$ is $h_{\inv}(y)=h_{\ww}(I_0(x):\cdots:I_r(x))$ for any semistable $x$ over $y$.

\begin{proposition}\label{prop:quotient}
$m\,h_{\inv}=h_{\Hq}+O(1)$ on $Y(\Kbar)$. If $f\colon A\to Y$ is a morphism from an abelian variety, then $h_{\inv}\circ f=\hhat_{f,\ww}+\frac1m\lambda_{N^-}+O(1)$ with $N=f^*\Hq$, and $\hhat_{f,\ww}=\lim_j4^{-j}h_{\inv}(f([2^j]P))$.
\end{proposition}

\begin{proof}
Combine Corollary~\ref{cor:pullback} with Theorem~\ref{thm:wnt}.
\end{proof}

\begin{proposition}\label{prop:metric}
Let $\overline\Hq_1,\overline\Hq_2$ be two adelic metrics on a line bundle $\Hq$ on a projective variety $Y/K$, continuous at the archimedean places and equal at all but finitely many places. For $y\in Y(F)$ and a local section $s$ of $\Hq$ not vanishing at $y$,
\[
h_{\overline\Hq_2}(y)-h_{\overline\Hq_1}(y)=\frac1{[F:\Q]}\sum_{u\in M_F}n_u\log\frac{\|s(y)\|_{1,u}}{\|s(y)\|_{2,u}},
\]
and this difference is bounded on $Y(\Kbar)$. In particular, for a morphism $f\colon A\to Y$ from an abelian variety, the limit of $4^{-j}h_{\overline\Hq}(f([2^j]P))$ does not depend on the adelic metric.
\end{proposition}

\begin{proof}
The local term is independent of $s$, since changing $s$ multiplies both norms by the same factor. At the finitely many places where the metrics differ, the ratio of the two norms is a continuous positive function on the compact analytic space $Y_u^{\rm an}$, hence bounded above and below.
\end{proof}

\subsection{Orbit-minimal heights}
The comparison in Proposition~\ref{prop:metric} is qualitative: it requires a metric on the quotient and gives no constant. For orbit-minimal heights one can do better, because an invariant is bounded on the whole orbit by the norm of any single point of the orbit.

Let now $G\subseteq\GL_{N+1}$ be reductive over $\Q$, acting on $V=\mathbf A^{N+1}$, and let $I_0,\ldots,I_r\in\Z[x_0,\ldots,x_N]$ be homogeneous $G$-invariants of degrees $w_i\ge1$ whose common zero locus in $V(\Qbar)$ is the nullcone $\{x:0\in\overline{Gx}\}$. Thus $x$ is semistable if and only if $I_i(x)\ne0$ for some $i$. For a polynomial $I$ let $\|I\|_1$ be the sum of the absolute values of its coefficients. For $x\in V(K)$ and $v\in M_K$ let $\|x\|_v=\max_j|x_j|_v$ (extended to $\Kbar_v$), and put
\[
\rho_v(x)=\inf_{g\in G(\Kbar_v)}\|gx\|_v,\qquad W_v(x)=\max_i|I_i(x)|_v^{1/w_i}.
\]
For semistable $x$ define
\begin{align*}
h_{\GIT}(x)&=\frac1{[K:\Q]}\sum_{v\in M_K}n_v\log\rho_v(x),\\
h_I(x)&=\frac1{[K:\Q]}\sum_{v\in M_K}n_v\log W_v(x)=h_{\ww}\bigl(I_0(x):\cdots:I_r(x)\bigr).
\end{align*}

\begin{theorem}\label{thm:gitgeneral}
Let $x\in V(K)$ be semistable and $c(I)=\max_i\frac1{w_i}\log\|I_i\|_1$.
\begin{enumerate}
\item $0<\rho_v(x)\le\|x\|_v$ for every $v$, and $\rho_v(x)=\|x\|_v=1$ for all but finitely many $v$. Hence $h_{\GIT}(x)$ is well defined, independent of $K$, and invariant under $G(\Qbar)$ and under scaling.
\item $W_v(x)\le\rho_v(x)$ for every finite $v$, and $W_v(x)\le\|I\|^{*}\rho_v(x)$ for every archimedean $v$, where $\|I\|^*=\max_i\|I_i\|_1^{1/w_i}$.
\item $h_I(x)\le h_{\GIT}(x)+c(I)$.
\item $h_{\GIT}(x)\le h(gx)$ for every $g\in G(\Qbar)$. In particular $h_I(x)\le h(gx)+c(I)$ for every $g\in G(\Qbar)$.
\item If $v$ is finite and $\|x\|_v=1$, then $0\le\log\rho_v(x)-\log W_v(x)\le-\log W_v(x)$. The local difference vanishes whenever some $I_i(x)$ is a $v$-adic unit.
\end{enumerate}
\end{theorem}

\begin{proof}
Since the $I_i$ are $G$-invariant polynomials, $I_i(gy)=I_i(y)$ for all $g\in G(\Kbar_v)$ and $y\in V(\Kbar_v)$. At an archimedean place $|I_i(y)|_v\le\|I_i\|_1\|y\|_v^{w_i}$ by the triangle inequality, and at a finite place $|I_i(y)|_v\le\|y\|_v^{w_i}$ because $I_i$ has integer coefficients. Applying this to $y=gx$ and taking the infimum over $g$ proves (2). As $x$ is semistable, some $I_i(x)\ne0$, so $W_v(x)>0$ and $\rho_v(x)>0$ by (2); taking $g=1$ gives $\rho_v(x)\le\|x\|_v$. For all but finitely many $v$ the vector $x$ is primitive at $v$ and a fixed nonzero $I_i(x)$ is a unit, so $1=W_v(x)\le\rho_v(x)\le\|x\|_v=1$. Independence of $K$ holds because $\rho_v$ depends only on the place of $\Kbar$; invariance under $G(\Qbar)$ is clear, and invariance under scaling follows from $\rho_v(\lambda x)=|\lambda|_v\rho_v(x)$ and the product formula. This proves (1). Summing (2) over all places, with $\sum_{v\mid\infty}n_v=[K:\Q]$, gives (3). For (4), $\rho_v(x)=\rho_v(gx)\le\|gx\|_v$ for every place. Finally (5) follows from (2) and $\rho_v(x)\le1$.
\end{proof}

\begin{remark}\label{rem:reverse}
The inequality in (3) is effective. By Zhang's geometric reductivity at the archimedean places \cite{Zhang1994} and Burnol's results at the finite places \cite{Burnol}, the orbit-minimal norms define an adelic metric on $\OO_Y(m)$, so that $h_{\GIT}-h_I$ is bounded on the semistable locus; see also \cite{Zhang1996}. The upper bound for $h_{\GIT}-h_I$ obtained in this way is not explicit, and only the effective inequality (3) is used below. Replacing the maximum norm by another norm at finitely many places changes $h_{\GIT}$ by a bounded amount. Zhang's normalized height of semistable cycles uses the Chow line and a normalization involving degree and dimension \cite{Zhang1996,ShaskaShaska}; to compare it with $h_I$ one has to identify the two metrized line bundles on $Y$, after which Proposition~\ref{prop:metric} applies.
\end{remark}

\section{Binary sextics and the moduli of genus-two curves}\label{sec:sextics}

We now specialize to binary sextics, the case relevant to curves of genus two. The main tasks are to fix a normalization of the invariants once and for all, to compute the constants of Theorem~\ref{thm:gitgeneral}, and to record the facts about the quotient used later.

\subsection{Igusa--Clebsch invariants}
For binary forms $P,Q$ of degrees $m,n$ the normalized $r$-th transvectant is
\[
(P,Q)_r=\frac{(m-r)!\,(n-r)!}{m!\,n!}\sum_{j=0}^r(-1)^j\binom rj\frac{\partial^rP}{\partial X^{r-j}\partial Z^j}\frac{\partial^rQ}{\partial X^j\partial Z^{r-j}}.
\]
For a sextic $F=\sum_{i=0}^6a_iX^{6-i}Z^i$ put $\mathbf H=(F,F)_4$, $\mathbf A=(F,F)_6$, $\mathbf B=(\mathbf H,\mathbf H)_4$, $\mathbf C=(\mathbf H,(\mathbf H,\mathbf H)_2)_4$, and
\begin{equation}\label{eq:IC}
\begin{aligned}
I_2&=-120\mathbf A,& I_4&=-720\mathbf A^2+6750\mathbf B,\\
I_6&=8640\mathbf A^3-108000\mathbf A\mathbf B+202500\mathbf C,& I_{10}&=\Disc(F),
\end{aligned}
\end{equation}
where $\Disc(F)$ is the discriminant of the binary form, equal to $\Disc(F(x,1))$ when $a_0\ne0$. These are the Igusa--Clebsch invariants in the normalization of Mestre \cite{Mestre}, which is also the one used by Sage and Magma. Explicitly $I_2=-240a_0a_6+40a_1a_5-16a_2a_4+6a_3^2$. In terms of the roots $\alpha_1,\ldots,\alpha_6$ of $F(x,1)$, the invariant $I_2$ is $a_0^2\sum(\alpha_1-\alpha_2)^2(\alpha_3-\alpha_4)^2(\alpha_5-\alpha_6)^2$, summed over the fifteen partitions into pairs, and similarly for $I_4,I_6$ \cite{Igusa}; we checked that both descriptions agree on $f_0$ in Appendix~\ref{app:data}. The integral Igusa invariants are
\begin{gather*}
J_2=\tfrac18I_2,\qquad J_4=\tfrac1{96}(4J_2^2-I_4),\qquad J_6=\tfrac1{576}(8J_2^3-160J_2J_4-I_6),\\
J_8=\tfrac14(J_2J_6-J_4^2),\qquad J_{10}=\tfrac1{4096}I_{10}.
\end{gather*}

\begin{proposition}\label{prop:ICnorms}
The invariants $I_2,I_4,I_6,I_{10}$ are polynomials with integer coefficients in $a_0,\ldots,a_6$, with $4$, $16$, $56$ and $246$ monomials and
\[
\|I_2\|_1=302,\quad\|I_4\|_1=4192,\quad\|I_6\|_1=283842,\quad\|I_{10}\|_1=1561993.
\]
Their common zero locus in $\Sym^6$ over $\Qbar$ is the set of sextics with a root of multiplicity at least four.
\end{proposition}

\begin{proof}
The first part is a direct expansion of \eqref{eq:IC}. The even-degree invariant ring of binary sextics in characteristic zero is generated by $I_2,I_4,I_6,I_{10}$, and the square of the odd invariant $I_{15}$ is a polynomial in them without constant term \cite{Igusa,Ottaviani}. Hence the four invariants vanish simultaneously exactly on the nullcone, which for sextics consists of the forms with a root of multiplicity at least four \cite{MumfordGIT}.
\end{proof}

\subsection{Weighted versus GIT height for binary sextics}
Let $\SL_2$ act on sextics by $(gF)(X,Z)=F((X,Z)g)$. With $\|F\|_v$ the maximum of $|a_i|_v$, let $\rho_v$, $W_v$, $h_{\GIT}$ and $h_{\inv}=h_I$ be as in Section~\ref{sec:git}. We say that $F$ has \emph{good reduction} at a finite place $v$ if, after scaling $F$ to be primitive at $v$, its discriminant is a $v$-adic unit; when $v\nmid2$ this means that the curve $y^2=F(x,1)$ has good reduction at $v$. With these conventions Theorem~\ref{thm:B} is a direct consequence of Theorem~\ref{thm:gitgeneral} and Proposition~\ref{prop:ICnorms}.

\begin{proof}[Proof of Theorem~\ref{thm:B}]
By Proposition~\ref{prop:ICnorms} we may apply Theorem~\ref{thm:gitgeneral}, with $c(I)=\max\{\frac12\log302,\frac14\log4192,\frac16\log283842,\frac1{10}\log1561993\}=\frac12\log302$. This gives (1) and (3) for $g\in\SL_2(\Qbar)$; for $g\in\GL_2(\Qbar)$ write $g=\lambda g_1$ with $g_1\in\SL_2(\Qbar)$, so that $gF=\lambda^6g_1F$ has the same projective height. For (2), if $F$ is primitive at $v$ and has good reduction there, then $I_{10}(F)=\Disc(F)$ is a $v$-adic unit and Theorem~\ref{thm:gitgeneral}(5) applies.
\end{proof}

\begin{remark}\label{rem:bombieri}
The archimedean norm in $h_{\GIT}$ can be replaced by the $\mathrm U(2)$-invariant norm $\|F\|_B=(\sum_i|a_i|^2/\binom6i)^{1/2}$. Since $\|F\|_\infty\le\sqrt{20}\,\|F\|_B$ and $\|F\|_B\le\sqrt7\,\|F\|_\infty$, the corresponding height $h_{\GIT}^B$ satisfies $h_{\inv}\le h_{\GIT}^B+\frac12\log6040$.
\end{remark}

\subsection{The moduli quotient}
We end the section with the quotient itself and with an observation that explains the shape of the families in Section~\ref{sec:surface}. Let $k$ be a number field. The even-degree invariant ring gives
\begin{equation}\label{eq:Q}
\Xq=\Proj k[I_2,I_4,I_6,I_{10}]\simeq\PP(2,4,6,10),\qquad\Hq=\OO_{\Xq}(60),
\end{equation}
so that $60\,h_{\inv}=h_{\Hq}+O(1)$ on $\Xq(\Kbar)$. A sextic is semistable if it has no root of multiplicity greater than three and stable if it has no root of multiplicity at least three. The open subset $\mathcal M_2=\{I_{10}\ne0\}$ is the coarse moduli space of smooth genus-two curves \cite{Igusa}; for weighted moduli of sextics see also \cite{BeshajGuest}.

\begin{proposition}\label{prop:affine}
Every morphism from a connected projective variety to $\mathcal M_2$ is constant. Hence every nonconstant morphism from an abelian variety to $\Xq$ whose image meets $\mathcal M_2$ also meets the boundary $I_{10}=0$.
\end{proposition}

\begin{proof}
$\mathcal M_2=D_+(I_{10})$ is affine, and global functions on a connected projective variety are constant.
\end{proof}

The weighted height of a moduli point and the N\'eron--Tate height of a point of the corresponding Jacobian live on different spaces. They can be compared only after one fixes a nonconstant family over a parameter variety together with a section of its relative Jacobian; this is what the remaining sections do.

\section{A pencil with five fixed branch points}\label{sec:pencil}

Before turning to two-dimensional families we work out a one-parameter example in which every quantity can be computed by hand. Fix five of the six branch points and let the sixth move along the line. The resulting pencil has explicit invariants, a moduli map that pulls the polarization back to the smallest possible degree, and an exact place-by-place formula relating the moduli height to the height of the parameter.

For $(T:U)\in\PP^1$ consider
\begin{gather*}
F_{T,U}(X,Z)=X(X-Z)(X-2Z)(X-3Z)(UX-TZ)Z,\\
C_t:\ y^2=x(x-1)(x-2)(x-3)(x-t),
\end{gather*}
with $t=T/U$. At each of $t=0,1,2,3,\infty$ the form $F_{T,U}$ has exactly one double root and no triple root, so every member of the pencil is stable. Put
\[
Q_2=4T^2-12TU+15U^2,\qquad D_5=UT(T-U)(T-2U)(T-3U).
\]

\begin{proposition}\label{prop:pencil}
The Igusa--Clebsch invariants of $F_{T,U}$ are
\begin{align*}
I_2&=10\,Q_2,\qquad I_{10}=144\,D_5^2,\\
I_4&=4\,(13T^4-78T^3U+183T^2U^2-198TU^3+108U^4),\\
I_6&=120\,(5T^6-45T^5U+184T^4U^2-429T^3U^3\\
&\qquad+599T^2U^4-474TU^5+180U^6).
\end{align*}
The forms $Q_2$ and $D_5$ have no common zero over $\Qbar$, and none over $\overline\F_p$ for $p\notin\{2,3,5,7\}$. The curve $C_t$ is smooth exactly when $D_5(T,U)\ne0$.
\end{proposition}

\begin{proof}
The invariants are obtained by substituting the coefficients
\[
\bigl(0,\ U,\ -(T+6U),\ 6T+11U,\ -(11T+6U),\ 6T,\ 0\bigr)
\]
into \eqref{eq:IC}. The discriminant is the product of the squared $2\times2$ determinants of the six linear factors, which equals $144D_5^2$. The zeros of $D_5$ are $(1:0),(0:1),(1:1),(2:1),(3:1)$, where $Q_2$ takes the values $4,15,7,7,15$; their prime factors lie in $\{2,3,5,7\}$.
\end{proof}

Let $g\colon\PP^1\to\Xq$ be the moduli map of the pencil and $g_{2,10}\colon\PP^1\to\PP(2,10)$, $(T:U)\mapsto(Q_2:D_5^2)$, the two-invariant projection along the pencil.

\begin{theorem}\label{thm:pullback}
$g$ is a nonconstant morphism with $g^*\Hq\simeq\OO_{\PP^1}(60)$, and $g_{2,10}^*\OO(10)\simeq\OO_{\PP^1}(10)$. Hence $h_{\inv}(g(t))=h(t)+O(1)$ and $h_{(2,10)}(g_{2,10}(t))=h(t)+O(1)$ on $\PP^1(\Qbar)$.
\end{theorem}

\begin{proof}
The coefficients of $F_{T,U}$ are linear forms, so an invariant of degree $d$ is a section of $\OO(d)$. By Proposition~\ref{prop:pencil} the sections $I_2$ and $I_{10}$ have no common zero, so the invariant sections are base-point free and the weighted monomials of degree $60$ (respectively $10$) generate $\OO(60)$ (respectively $\OO(10)$). The map $g$ is nonconstant because $I_{10}$ vanishes at some but not all parameters.
\end{proof}

Theorem~\ref{thm:pullback} compares the two heights only up to $O(1)$. For the two-invariant height the difference can be written down exactly, prime by prime. For primitive $(T,U)\in\Z^2$ put $H(T,U)=\log\max\{|T|,|U|\}$, $H_*(T,U)=h_{(2,10)}(Q_2:D_5^2)$,
\[
\epsilon_\infty=\log\frac{\max\{|Q_2|^{1/2},|D_5|^{1/5}\}}{\max\{|T|,|U|\}},\qquad\epsilon_p=-\log p\cdot\min\Bigl\{\frac{v_p(Q_2)}2,\frac{v_p(D_5)}5\Bigr\}.
\]

\begin{theorem}\label{thm:local}
For every primitive $(T,U)\in\Z^2$,
\begin{equation}\label{eq:exactlocal}
H_*(T,U)=H(T,U)+\epsilon_\infty(T,U)+\sum_{p\in\{2,3,5,7\}}\epsilon_p(T,U),
\end{equation}
with
\begin{gather*}
-\log2\le\epsilon_2\le0,\qquad-\tfrac12\log p\le\epsilon_p\le0\quad(p=3,5,7),\\
\tfrac12\log\tfrac85\le\epsilon_\infty\le\tfrac12\log31.
\end{gather*}
The bounds are sharp: $\epsilon_2=-\log2$ whenever $2^5\mid U$, $\epsilon_\infty=\frac12\log\frac85$ at $(T,U)=(5,2)$, and $\epsilon_\infty=\frac12\log31$ at $(T,U)=(1,-1)$.
\end{theorem}

\begin{proof}
At a prime $p$ the local weighted norm is $\max\{|Q_2|_p^{1/2},|D_5|_p^{1/5}\}$, whose logarithm is $\epsilon_p$; by Proposition~\ref{prop:pencil} it is $0$ for $p\notin\{2,3,5,7\}$. For $p=2$, $Q_2\equiv U^2\pmod2$, so a correction requires $U$ even and $T$ odd; then $v_2(D_5)=v_2(U)$, while $v_2(Q_2)=3$ if $v_2(U)=1$ and $v_2(Q_2)=2$ if $v_2(U)\ge2$, so the minimum is at most $1$, with equality when $v_2(U)\ge5$. For $p=3$, $Q_2\equiv T^2\pmod3$ forces $3\mid T$, $3\nmid U$, and then $v_3(Q_2)=1$. For $p=5$ the common zeros modulo $5$ are $T\equiv0$ and $T\equiv3U$, and for $p=7$ they are $T\equiv U$ and $T\equiv2U$. If $a$ is the relevant root and $v_p(T-aU)=1$ then $v_p(D_5)=1$; if $v_p(T-aU)\ge2$ then $Q_2\equiv Q_2(a,1)U^2\pmod{p^2}$ with $Q_2(a,1)\in\{15,7\}$ of valuation one. At the real place $Q_2=(2T-3U)^2+6U^2$; on the square $\max\{|T|,|U|\}=1$ its minimum is $\frac85$, at $(1,\frac25)$, and its maximum is $31$, while $|D_5|\le24$ and $24^{1/5}<\sqrt{31}$.
\end{proof}

\begin{remark}\label{rem:fullIC}
For the full Igusa--Clebsch point, $h_{\inv}(g(t))-H(T,U)$ is a sum of an archimedean term and local terms at $2,3,5,7$ only: at any other prime either $I_2=10Q_2$ or $I_{10}=144D_5^2$ is a unit, and all four invariants are integral.
\end{remark}

\subsection*{An elliptic parameter}
Composing the pencil with the $x$-coordinate of an elliptic curve turns it into a family over an abelian variety, to which the results of Section~\ref{sec:abelian} apply. Let $x\colon E\to\PP^1$ be the $x$-coordinate of $E:y^2=x^3-2$, and $f=g\circ x$, $f_*=g_{2,10}\circ x$.

\begin{theorem}\label{thm:ell}
$f^*\Hq\simeq\OO_E(120O)$ and $f_*^*\OO(10)\simeq\OO_E(20O)$. For $R\in E(\Qbar)$,
\[
h_{\inv}(f(R))=2\hhat_E(R)+O(1),\qquad H_*(x(R))=2\hhat_E(R)+O(1),
\]
and the weighted N\'eron--Tate heights of $f$ and $f_*$ both equal $2\hhat_E$. Consequently
\[
\#\{n\in\Z:\ h_{\inv}(f([n]P_0))\le X\}=\frac{2}{\sqrt{2\hhat_E(P_0)}}\sqrt X+O(1).
\]
\end{theorem}

\begin{proof}
Since $x^*\OO(1)=\OO_E(2O)$, the pullbacks follow from Theorem~\ref{thm:pullback}, and $h(x(R))=2\hhat_E(R)+O(1)$. The pullback bundles are symmetric, so Theorems~\ref{thm:decomposition} and~\ref{thm:wnt} give the weighted N\'eron--Tate heights, and the count follows as in Theorem~\ref{thm:count}.
\end{proof}

\section{Families of binary sextics over \texorpdfstring{$E\times E$}{E x E}}\label{sec:surface}

The pencil of Section~\ref{sec:pencil} has a rational base, and its parameter space has no group structure of its own. To compare heights along subgroups we need a parameter space that is itself an abelian variety. By Proposition~\ref{prop:affine} such a family must degenerate, so the best one can hope for is a morphism to the compactified quotient $\Xq$ whose image meets $\mathcal M_2$. We show first that such families exist in abundance, and then construct two explicit ones.

Let $A=E\times E$ and $L=p_1^*\OO_E(3O)\otimes p_2^*\OO_E(3O)$. The bundle $L$ is symmetric and ample, $h^0(A,L)=9$, and $h_L(P_1,P_2)=3\hhat_E(P_1)+3\hhat_E(P_2)+O(1)$.

\subsection{An existence theorem}

\begin{lemma}\label{lem:codim}
The sextics with a root of multiplicity at least four form a closed subset of $\PP(\Sym^6)\simeq\PP^6$ of dimension three.
\end{lemma}

\begin{proof}
It is the image of $\PP^1\times\PP^2\to\PP^6$, $(\ell,q)\mapsto\ell^4q$, which has finite general fibres.
\end{proof}

\begin{theorem}\label{thm:surface}
There is a nonempty Zariski-open subset $\mathcal U$ of $H^0(A,L)^{\oplus7}\simeq\mathbf A^{63}_\Q$ with the following property. For $\boldsymbol s=(s_0,\ldots,s_6)\in\mathcal U(\Q)$ every sextic $F_R=\sum_js_j(R)X^{6-j}Z^j$ with $R\in A$ is semistable, the general one is smooth, and the moduli morphism $\Phi_{\boldsymbol s}\colon A\to\Xq$ is finite onto a surface with $\Phi_{\boldsymbol s}^*\Hq\simeq L^{\otimes60}$. Moreover
\[
h_{\inv}(\Phi_{\boldsymbol s}(R))=3\hhat_E(p_1R)+3\hhat_E(p_2R)+O_{\boldsymbol s}(1).
\]
\end{theorem}

\begin{proof}
Let $W=H^0(A,L)^{\oplus7}$. As $L$ is globally generated, evaluation at any $R$ is surjective onto $L_R^{\oplus7}$, so the tuples whose sextic at $R$ is unstable or zero form a subset of codimension three by Lemma~\ref{lem:codim}. The incidence variety of pairs $(R,\boldsymbol s)$ with $F_R$ unstable or zero has dimension at most $63+2-3=62$, so its image in $W$ is a proper closed subset. The tuples with $\Disc(F_{R_0})\ne0$ for a fixed $R_0$ form a nonempty open set, and $\mathcal U$ is the intersection; it has rational points because $\Q$ is infinite. For $\boldsymbol s\in\mathcal U$ the invariant $I_d(F_R)$ is a section of $L^{\otimes d}$, the invariant sections have no common zero, and the weighted monomials of degree $60$ identify $\Phi_{\boldsymbol s}^*\Hq$ with $L^{\otimes60}$. Ampleness of the pullback rules out contracted curves, so $\Phi_{\boldsymbol s}$ is finite. Since $L$ is symmetric, Theorem~\ref{thm:decomposition} has no linear term, and $h_{\inv}\circ\Phi_{\boldsymbol s}=\frac1{60}h_{L^{\otimes60}}+O(1)=h_L+O(1)$.
\end{proof}

\begin{corollary}\label{cor:ranktwo}
For $\boldsymbol s\in\mathcal U(\Q)$,
\[
\#\{(n,m)\in\Z^2:\ h_{\inv}(\Phi_{\boldsymbol s}([n]P_0,[m]P_0))\le X\}=\frac{\pi}{3\hhat_E(P_0)}X+O_{\boldsymbol s}(\sqrt X),
\]
and the same asymptotic holds after restricting to smooth fibres. The number of distinct smooth genus-two moduli points obtained is of exact order $X$.
\end{corollary}

\begin{proof}
The main term comes from lattice points in the disks $3\hhat_E(P_0)(n^2+m^2)\le X+O(1)$. The discriminant divisor $D_{\rm sing}\subset A$ is a proper closed curve, so by Faltings' theorem \cite{Faltings1991,Faltings1994} its intersection with $\Gamma=\Z P_0\times\Z P_0$ is a finite union of cosets of subgroups of rank at most one, contributing $O(\sqrt X)$ points. The subgroups have rank at most one because $\Gamma$ is Zariski dense in $A$, since $P_0$ has infinite order, while $D_{\rm sing}$ is a curve. Since $\Phi_{\boldsymbol s}$ is finite, its geometric fibres have bounded cardinality.
\end{proof}

\subsection{A base-point-free net and the parameter plane}
Theorem~\ref{thm:surface} is an existence statement and does not produce a single explicit family. We now build two, starting from a net of sections of $L$. Write $E:Y^2Z=X^3-2Z^3$, $O=(0:1:0)$, and use coordinates $(X_i:Y_i:Z_i)$ on the two factors. Define sections of $L$ by
\begin{equation}\label{eq:abc}
a=Y_1Y_2,\qquad b=Z_1X_2+X_1Z_2,\qquad c=Z_1Z_2+X_1Y_2+Y_1X_2.
\end{equation}

\begin{lemma}\label{lem:abc}
The sections $a,b,c$ have no common zero on $A_{\Qbar}$.
\end{lemma}

\begin{proof}
At $(O,O)$ we have $a=1$. If exactly one factor is $O$, say the first, then $(a,b,c)$ is proportional to $(Y_2,0,X_2)$, and $X_2=Y_2=0$ is impossible on $E$ when $2$ is invertible. In the affine chart a common zero satisfies $y_1y_2=0$, $x_2=-x_1$ and $1+x_1y_2+y_1x_2=0$. If $y_1=0$ then $x_1^3=2$, $y_2^2=x_2^3-2=-4$ and $x_1y_2=-1$, whence $x_1^2=-\frac14$ and $x_1^6=-\frac1{64}$, contradicting $x_1^6=4$. The case $y_2=0$ is symmetric.
\end{proof}

The same computation shows that $a,b,c$ have no common zero on the reduction of $A$ modulo any prime $p\notin\{2,3,257\}$, the only new relation in characteristic $p$ being $256=-1$; we shall not need this.

\begin{proposition}\label{prop:net}
The map $\phi=(a:b:c)\colon A\to\PP^2$ is a finite surjective morphism of degree $18$ with $\phi^*\OO(1)\simeq L$ and $\phi(O,O)=e=(1:0:0)$. It is invariant under the involution $s(P_1,P_2)=(P_2,P_1)$.
\end{proposition}

\begin{proof}
By Lemma~\ref{lem:abc}, $\phi$ is a morphism with $\phi^*\OO(1)\simeq L$. A curve contracted by $\phi$ would have $L$-degree zero, which is impossible since $L$ is ample, so $\phi$ is finite and hence surjective. Its degree is $(L\cdot L)=18$, since $L\equiv3(O\times E)+3(E\times O)$ and $(O\times E)\cdot(E\times O)=1$. The formulas \eqref{eq:abc} are symmetric under exchanging the indices.
\end{proof}

All families below are pulled back from families of sextics over $\PP^2$ along $\phi$, and we write $\psi_1,\psi_2\colon\PP^2\to\Xq$ for the moduli maps of the two families, so that $\Phi_i=\psi_i\circ\phi$. The plane $\PP^2$ carries the geometry of the families, while the abelian surface supplies the group structure used for counting and for the elliptic curves of Section~\ref{sec:specialization}.

\subsection{The linear family}
The first family is linear in the net. Let
\begin{equation}\label{eq:linfamily}
F^{(1)}_{a,b,c}=a(X^6+Z^6)+b(X^5Z+XZ^5)+c(-X^5Z+X^4Z^2-2X^3Z^3+X^2Z^4),
\end{equation}
with coefficient vector $(a,\,b-c,\,c,\,-2c,\,c,\,b,\,a)$. In the basis
\[
(Z_1Z_2,\ Z_1X_2,\ Z_1Y_2,\ X_1Z_2,\ X_1X_2,\ X_1Y_2,\ Y_1Z_2,\ Y_1X_2,\ Y_1Y_2)
\]
of $H^0(A,L)$ its coefficient matrix is
\[
\begin{pmatrix}
0&0&0&0&0&0&0&0&1\\
-1&1&0&1&0&-1&0&-1&0\\
1&0&0&0&0&1&0&1&0\\
-2&0&0&0&0&-2&0&-2&0\\
1&0&0&0&0&1&0&1&0\\
0&1&0&1&0&0&0&0&0\\
0&0&0&0&0&0&0&0&1
\end{pmatrix},
\]
of rank three: the seven coefficients are linear combinations of $a$, $b$, $c$, which is why the family factors through $\phi$.

\begin{proposition}\label{prop:linear}
Every member of \eqref{eq:linfamily} with $(a:b:c)\in\PP^2(\Qbar)$ is semistable, and $F^{(1)}_{1,0,0}=X^6+Z^6$ is squarefree. The moduli map $\Phi_1=\psi_1\circ\phi\colon A\to\Xq$ is finite onto a surface, $\Phi_1^*\Hq\simeq L^{\otimes60}$, and $h_{\inv}(\Phi_1(R))=3\hhat_E(p_1R)+3\hhat_E(p_2R)+O(1)$. Its invariants are
\begin{align*}
I_2&=-8\,(30a^2-5b^2+5bc-c^2),\\
I_4&=4\,(405a^4-135a^2b^2+135a^2bc+198a^2c^2+150ab^2c+30abc^2\\
&\qquad+9ac^3-20b^4+40b^3c+17b^2c^2-25bc^3-5c^4),\\
I_6&=-4\,(29970a^6-14985a^4b^2+14985a^4bc+4878a^4c^2+9300a^3b^2c\\
&\qquad-6240a^3bc^2+3168a^3c^3+560a^2b^4+1130a^2b^3c-5051a^2b^2c^2\\
&\qquad+4735a^2bc^3-748a^2c^4-800ab^4c-260ab^3c^2+942ab^2c^3\\
&\qquad-182abc^4-60ac^5+80b^6-240b^5c+48b^4c^2\\
&\qquad+330b^3c^3-270b^2c^4+32bc^5+17c^6),
\end{align*}
and $I_{10}$ is the form of degree ten displayed in \eqref{eq:I10lin}.
\end{proposition}

\begin{proof}
In the chart $Z=1$ let $g_a=x^6+1$, $g_b=x^5+x$, $g_c=-x^5+x^4-2x^3+x^2$. A member with a root of multiplicity four at $x$ gives a nonzero kernel vector of the $4\times3$ matrix of derivatives of orders $0,\ldots,3$ of $(g_a,g_b,g_c)$ at $x$. The minors on the rows $(0,1,2)$ and $(1,2,3)$ are
\begin{gather*}
-2(x^{12}-6x^{11}+6x^{10}+10x^9-15x^8+30x^7\\
-30x^5+15x^4+10x^3-6x^2+6x-1),\\
-120x^3(2x^6-9x^5+6x^4+5x^3-6x^2+9x-2),
\end{gather*}
which have no common root. At $Z=0$ the jets of orders $0,1,2$ form a triangular matrix of rank three. The invariants are computed from \eqref{eq:IC}. The remaining statements follow as in Theorem~\ref{thm:surface}, since $\phi^*\OO(1)=L$ and the coefficients are linear in $(a,b,c)$.
\end{proof}

For the linear family the local comparison of Theorem~\ref{thm:B} can be made uniform in the parameter: the finite places at which $h_{\GIT}$ and $h_{\inv}$ can differ form a fixed finite set.

\begin{proposition}\label{prop:linear-git}
Let $(a:b:c)\in\PP^2(K)$ and let $v$ be a finite place of $K$ whose residue characteristic is not in $\{2,3,5,14447\}$. Then the local difference $\log\rho_v-\log W_v$ for $F^{(1)}_{a,b,c}$ vanishes. In particular, for every $R\in A(\Qbar)$ the difference $h_{\GIT}(F^{(1)}_R)-h_{\inv}(\Phi_1(R))$ is the sum of an archimedean term, which is at least $-\frac12\log302$, and of nonnegative terms at places above $2,3,5$ and $14447$. The prime $14447$ is genuinely exceptional: the member with $(a:b:c)\equiv(1:8832:13222)$ reduces modulo $14447$ to $(x+890)(x+4091)(x+1269)^4$.
\end{proposition}

\begin{proof}
Scale $(a,b,c)$ to be primitive at $v$; then the coefficient vector is primitive, since $a$, $b$ and $c$ all occur among its entries. By Theorem~\ref{thm:gitgeneral}(5) it suffices to show that $I_2,I_4,I_6,I_{10}$ of the family have no common zero on $\PP^2(\overline\F_p)$ for $p\notin\{2,3,5,14447\}$. In the chart $a=1$ let $r_k(b)=\Res_c(I_2,I_k)$ for $k=4,6,10$, computed with formal degrees. A common zero $(b_0,c_0)$ modulo $p$ gives a common root $b_0$ of the reductions of $r_4,r_6,r_{10}$, hence $p$ divides both $N_1=\Res_b(r_4,r_6)$ and $N_2=\Res_b(r_4,r_{10})$. The common prime factors of $N_1$ and $N_2$ are $2,3,5,14447$ (Appendix~\ref{app:data}). On the line $a=0$ the restriction of $I_2$ is $8(5b^2-5bc+c^2)$; setting $c=1$ and taking primitive parts, the resultants of $5b^2-5b+1$ with the restrictions of $I_4,I_6,I_{10}$ are $2^23^25^2\cdot31$, $3^45^2\cdot79$ and $11\cdot109\cdot6199$, which have no prime factor common to all three; at $(0:1:0)$ we have $I_2=40$ and $I_{10}=256$. The factorization modulo $14447$ is a direct computation, and the other statements follow from Theorem~\ref{thm:gitgeneral}.
\end{proof}

\subsection{The quadratic family}
The curves of the linear family have no evident rational points. The second family is designed to have two, the Weierstrass point $x=0$ and a point at infinity, and their difference will give the section of the relative Jacobian that we need. Let
\begin{equation}\label{eq:f0}
\begin{aligned}
f_0(x)&=x(x-1)(x-2)(x-3)(x+20)(x+5)\\
&=x^6+19x^5-39x^4-331x^3+950x^2-600x,
\end{aligned}
\end{equation}
and
\begin{multline}\label{eq:quadfamily}
F^{(2)}_{a,b,c}=a^2Z^6f_0(X/Z)+abXZ^5+acX^2Z^4+b^2(X^5Z+XZ^5)\\
+bc(X^4Z^2+X^2Z^4)+c^2(X^4Z^2+XZ^5),
\end{multline}
whose coefficients, from $X^6$ to $Z^6$, are
\begin{multline}\label{eq:quadcoeff}
\bigl(a^2,\ 19a^2+b^2,\ -39a^2+bc+c^2,\ -331a^2,\\
950a^2+ac+bc,\ -600a^2+ab+b^2+c^2,\ 0\bigr).
\end{multline}
Each member is divisible by $X$, and its leading coefficient is the square $a^2$. Thus for $a\ne0$ the curve $y^2=F^{(2)}_{a,b,c}(x,1)$ has the rational Weierstrass point $W=(0,0)$ and a rational point $P_\infty^+$ at infinity, characterized by $y/x^3\to a$. If $(a,b,c)$ is primitive at a prime, so is the coefficient vector.

\begin{theorem}\label{thm:quadratic}
Every member of \eqref{eq:quadfamily} with $(a:b:c)\in\PP^2(\Qbar)$ is semistable, and $F^{(2)}_{1,0,0}=Z^6f_0(X/Z)$ is squarefree. The moduli map $\Phi_2=\psi_2\circ\phi\colon A\to\Xq$ is finite onto a surface, $\psi_2^*\Hq\simeq\OO_{\PP^2}(120)$, $\Phi_2^*\Hq\simeq L^{\otimes120}$, and
\begin{equation}\label{eq:quad-height}
h_{\inv}(\Phi_2(R))=6\hhat_E(p_1R)+6\hhat_E(p_2R)+O(1).
\end{equation}
Over the open set where $a\Disc(F^{(2)})\ne0$, the class $\sigma=[P_\infty^+-W]$ is a section of the relative Jacobian.
\end{theorem}

\begin{proof}
In the chart $Z=1$ a root of multiplicity four at $x$ forces the four polynomials $F$, $\partial_xF$, $\partial_x^2F$ and $\partial_x^3F$ to vanish. In each of the charts $a=1$, $b=1$, $c=1$ these four polynomials generate the unit ideal of the polynomial ring over $\Q$ in $x$ and the two remaining coordinates; a reduced Gr\"obner basis computation shows this (Appendix~\ref{app:data}). A root of multiplicity four at infinity forces the three leading coefficients $a^2$, $19a^2+b^2$ and $-39a^2+bc+c^2$ to vanish, hence $a=b=c=0$. The coefficients are quadratic in $(a,b,c)$, so $I_d$ pulls back to a section of $\OO(2d)$ on $\PP^2$, which gives $\psi_2^*\Hq\simeq\OO(120)$; finiteness follows from ampleness as before, and \eqref{eq:quad-height} from the symmetry of $L$.
\end{proof}

\begin{remark}\label{rem:unstable-reduction}
Semistability in characteristic zero does not survive reduction everywhere. The member $(a:b:c)\equiv(1:171:376)$ reduces modulo $587$ to $x(x+305)(x+48)^4$, and $(1:1300:666)$ reduces modulo $2341$ to $x(x+806)(x+338)^4$. Similarly, the stable form $\prod_{j=0}^5(X-5jZ)$ is unstable modulo $5$.
\end{remark}

\section{Sections of the relative Jacobian}\label{sec:sections}

A section of a relative Jacobian is useful for comparing heights only if it has infinite order. In this section we prove that the section $\sigma$ of Theorem~\ref{thm:quadratic} has infinite order, first over the function field and then on every fibre of an explicit infinite lattice of parameters, and we give explicit fibres of small height.

\subsection{Torsion and rational points}
The existence of a section of the relative Jacobian is not enough to produce positive heights. For instance, on $y^2=x(x-1)(x-2)(x-3)(x-t)$ over $\Q(t)$ we have $\divisor(x/(x-1))=2W_0-2W_1$ with $W_0=(0,0)$ and $W_1=(1,0)$, so $[W_0-W_1]$ has exact order two and height zero. Differences of Weierstrass points are always $2$-torsion, and the class $\sigma=[P_\infty^+-W]$ was chosen to avoid this.

Rational points on two-parameter families are easy to produce. For $a_0\in\Q\setminus\{0,1,2,3\}$, $c_0=a_0(a_0-1)(a_0-2)(a_0-3)$, $u=a_0-c_0s^2$ and $v=a_0-t^2$, the curve $y^2=x(x-1)(x-2)(x-3)(x-u)(x-v)$ over $\Q(s,t)$ contains $(a_0,c_0st)$, because the right side at $x=a_0$ equals $c_0\cdot c_0s^2\cdot t^2$. We shall not use this family. It illustrates that rational points are cheap, whereas proving that the resulting classes have infinite order is the real difficulty; the rest of this section addresses it for the family $F^{(2)}$.

\subsection{Torsion criteria}
Both arguments below rest on the injectivity of torsion under reduction at a prime of good reduction, which we use in the following two forms.

\begin{lemma}\label{lem:criteria}
Let $J/\Q$ be an abelian variety and let $p,q\ge3$ be primes of good reduction.
\begin{enumerate}
\item If $\ell$ is a prime with $\ell\nmid\#J(\F_q)$ and $D\in J(\Q)$ is such that $\ell$ divides the order of the reduction of $D$ in $J(\F_p)$, then $D$ has infinite order.
\item If $p_1,\ldots,p_k\ge3$ are primes of good reduction with $\gcd_i\#J(\F_{p_i})=1$, then $J(\Q)_{\tors}=0$.
\end{enumerate}
\end{lemma}

\begin{proof}
For a prime $p\ge3$ of good reduction the reduction map is injective on $J(\Q)_{\tors}$, because the ramification index $1$ is smaller than $p-1$ \cite[Appendix]{Katz}. This proves (2). For (1), if $D$ had finite order $n$, then $\ell$ would divide $n$ since the reduction of $D$ has order dividing $n$, and $\frac n\ell D$ would be a rational point of order $\ell$; it would inject into $J(\F_q)$, contradicting $\ell\nmid\#J(\F_q)$.
\end{proof}

\subsection{The reference curve and the congruence lattice}
We first treat the fibre over the base point and then propagate the result to a lattice of parameters that reduce to the base point modulo two primes. Let $C_0:y^2=f_0(x)$, so that $C_0$ is the fibre over $e=(1:0:0)$, and let $J_0=\Jac(C_0)$.

\begin{theorem}\label{thm:nontorsion}
The class $D=[P_\infty^+-W]\in J_0(\Q)$ has infinite order. Consequently the section $\sigma$ of Theorem~\ref{thm:quadratic} has infinite order over the function field of $A$, and over that of $\PP^2$.
\end{theorem}

\begin{proof}
The six roots of $f_0$ are distinct modulo $13$ and $17$, so $C_0$ has good reduction there. Counting points gives
\[
\begin{array}{c|ccc}
p&\#C_0(\F_p)&\#C_0(\F_{p^2})&\#J_0(\F_p)\\\hline
13&14&182&176=2^4\cdot11\\
17&18&286&288=2^5\cdot3^2
\end{array}
\]
Here, for a curve of genus two with $N_1$ points over $\F_p$ and $N_2$ points over $\F_{p^2}$, we put $c_1=N_1-p-1$ and $c_2=\frac12(c_1^2+N_2-p^2-1)$; the characteristic polynomial of Frobenius is $P_p(T)=T^4+c_1T^3+c_2T^2+pc_1T+p^2$ and $\#J(\F_p)=P_p(1)$. All Weierstrass points of $C_0$ remain rational and distinct modulo $13$, so $J_0(\F_{13})[2]$ has order $16$ and is the $2$-Sylow subgroup of $J_0(\F_{13})$. If the reduction of $D$ were killed by $2$, then $2P_\infty^+\sim2W$, and uniqueness of the hyperelliptic $g^1_2$ would force $P_\infty^+$ to be a Weierstrass point, which it is not. Hence $11$ divides the order of $D$ modulo $13$, and Lemma~\ref{lem:criteria}(1) with $\ell=11$, $p=13$, $q=17$ shows that $D$ has infinite order. A torsion section would specialize to a torsion class at $(O,O)$.
\end{proof}

\begin{theorem}\label{thm:grid}
Let $M=342$ and $R_{n,m}=([Mn]P_0,[Mm]P_0)$ for $(n,m)\in\Z^2$. For every $(n,m)$ the curve $C_{n,m}:y^2=F^{(2)}_{R_{n,m}}(x,1)$ is smooth, the class $\sigma(R_{n,m})$ has infinite order, so that $\Jac(C_{n,m})(\Q)$ has positive rank, and the local GIT corrections of $F^{(2)}_{R_{n,m}}$ at $13$ and $17$ vanish. Moreover
\begin{equation}\label{eq:grid-asymptotic}
\#\{(n,m)\in\Z^2:\ h_{\inv}(\Phi_2(R_{n,m}))\le X\}=\frac{\pi}{6M^2\hhat_E(P_0)}X+O(\sqrt X),
\end{equation}
and the number of distinct moduli points among these curves with $h_{\inv}\le X$ is $\gg X$.
\end{theorem}

\begin{proof}
The curve $E$ has good reduction at $13$ and $17$, with $\#E(\F_{13})=19$ and $\#E(\F_{17})=18$. As $M=19\cdot18$, every $R_{n,m}$ reduces to $(O,O)$ at both primes, so after primitive scaling $(a:b:c)\equiv(1:0:0)$ and $F^{(2)}_{R_{n,m}}\equiv a^2f_0$ modulo $13$ and $17$ with $a$ a unit. Hence $C_{n,m}$ has good reduction at $13$ and $17$ with reduction isomorphic to that of $C_0$, and $\sigma(R_{n,m})$ reduces to the reduction of $D$. The proof of Theorem~\ref{thm:nontorsion} applies verbatim to $\sigma(R_{n,m})$, without any information on the torsion of $\Jac(C_{n,m})(\Q)$. Good reduction makes the local GIT corrections vanish by Theorem~\ref{thm:B}(2). By \eqref{eq:quad-height}, $h_{\inv}(\Phi_2(R_{n,m}))=6M^2\hhat_E(P_0)(n^2+m^2)+O(1)$, and counting lattice points in disks gives \eqref{eq:grid-asymptotic}. The fibres of the finite map $\Phi_2$ have bounded cardinality.
\end{proof}

\subsection{Explicit fibres of small height}
The parameters of Theorem~\ref{thm:grid} have very large height: already the $x$-coordinate of $[342]P_0$ has about $6.9\cdot10^4$ decimal digits. Lemma~\ref{lem:criteria}(2) handles fibres of small height directly.

\begin{proposition}\label{prop:small}
For each $R$ in the table below the curve $C_R:y^2=F^{(2)}_R(x,1)$ is smooth, $\Jac(C_R)(\Q)$ is torsion-free, and $\sigma(R)$ has infinite order. In particular $\Jac(C_R)(\Q)$ has positive rank.
\[
\begin{array}{llll}
\toprule
R&(a:b:c)&\text{primes }p&\#\Jac(C_R)(\F_p)\\
\midrule
(P_0,O),\ (O,P_0)&(5:0:3)&7,\ 47&39,\ 2518\\
(P_0,P_0)&(25:6:31)&3,\ 17&20,\ 329\\
(P_0,-P_0),\ (-P_0,P_0)&(-25:6:1)&3,\ 5&17,\ 24\\
(-P_0,O),\ (O,-P_0)&(-5:0:3)&7,\ 11&81,\ 181\\
(-P_0,-P_0)&(25:6:-29)&3,\ 7&20,\ 77\\
([2]P_0,O),\ (O,[2]P_0)&(-383:0:1290)&11,\ 17&76,\ 321\\
\bottomrule
\end{array}
\]
For example, $R=(P_0,O)$ gives $C_R:y^2=25x^6+475x^5-966x^4-8275x^3+23765x^2-14991x$.
\end{proposition}

\begin{proof}
The values of $(a:b:c)$ follow from \eqref{eq:abc} with $P_0=(3:5:1)$, $-P_0=(3:-5:1)$, $[2]P_0=(1290:-383:1000)$ and $O=(0:1:0)$; the symmetry of Proposition~\ref{prop:net} explains the pairs. In each case the discriminant of the primitive sextic is nonzero and prime to the two listed primes, and the two group orders are coprime, so $\Jac(C_R)(\Q)_{\tors}=0$ by Lemma~\ref{lem:criteria}(2). Since $P_\infty^+\ne W$, the class $\sigma(R)$ is nonzero, hence of infinite order.
\end{proof}

\section{N\'eron--Tate heights along curves in the parameter plane}\label{sec:specialization}

We can now compare the N\'eron--Tate height of the section with the weighted moduli height. The comparison is made along curves in the parameter plane, where the specialization theory of heights applies. Besides Theorem~\ref{thm:nontorsion}, two inputs are needed: the family must have trivial trace along the curve, and the curve must pass through the base point $e$, where we control the fibre completely.

\subsection{Absolute simplicity of the reference Jacobian}
Triviality of the trace will follow from the absolute simplicity of the fibre over $e$, which we check by reduction modulo $31$.

\begin{lemma}\label{lem:simple31}
The Jacobian $J_0$ of $C_0:y^2=f_0(x)$ is absolutely simple.
\end{lemma}

\begin{proof}
The roots of $f_0$ are distinct modulo $31$, and $\#C_0(\F_{31})=36$, $\#C_0(\F_{31^2})=942$. Hence the characteristic polynomial of Frobenius is
\begin{equation}\label{eq:frob31}
P_{31}(T)=T^4+4T^3-2T^2+124T+31^2,
\end{equation}
and $\#J_0(\F_{31})=P_{31}(1)=1088$. The reduction is ordinary because $31\nmid-2$, and $P_{31}$ is irreducible over $\Q$. Moreover
\begin{multline*}
\Res_X\bigl(P_{31}(X),P_{31}(TX)\bigr)=31^4(T-1)^4(31T^4+64T^3+82T^2+64T+31)^2\\
\cdot(961T^4-620T^3+934T^2-620T+961),
\end{multline*}
and both quartic factors are irreducible with leading coefficients $31$ and $961$, so they have no roots of unity as zeros. Hence no quotient of two distinct Frobenius eigenvalues is a root of unity. Consequently, for every $r\ge1$ the $r$-th powers of the eigenvalues are four distinct conjugates, and the characteristic polynomial of the $r$-th power of Frobenius is irreducible over $\Q$. The reduction of $J_0$ is therefore simple over every finite extension of $\F_{31}$, since an isogeny decomposition over $\F_{31^r}$ would factor this polynomial; compare \cite{HoweZhu}. An abelian subvariety of $J_{0,\Qbar}$ would be defined over a number field and would have good reduction at the primes above $31$, giving a proper abelian subvariety of the reduction. Hence $J_0$ is absolutely simple.
\end{proof}

\subsection{Trivial trace}
We now fix the setting for the rest of the section. Let $B$ be a smooth projective geometrically connected curve over a number field $k$ and $\beta\colon B\to\PP^2$ a nonconstant morphism with $e\in\beta(B(\Qbar))$. Let $Z\subset\PP^2$ be the plane curve $a\cdot\Disc(F^{(2)}_{a,b,c})=0$ and $B^\circ=B\setminus\beta^{-1}(Z)$; this open subset contains every point over $e$. Over $B^\circ$ we have the smooth family $y^2=F^{(2)}_{\beta(b)}(x,1)$, its relative Jacobian $\Jfam_\beta$ with the principal polarization $\Theta_W$, and the section $\sigma_\beta=\beta^*\sigma$. Put $\mathcal K=\Qbar(B)$ and $d_\beta=\deg\beta^*\OO(1)$.

\begin{proposition}\label{prop:trace}
The generic fibre $\Jfam_{\beta,\eta}$ has trivial $\mathcal K/\Qbar$-trace, $\sigma_\beta$ has infinite order, and its N\'eron--Tate height $\delta_\beta=\hhat_{\mathcal K}(\sigma_\beta)$ over $\mathcal K$, normalized by degrees on $B$, is positive.
\end{proposition}

\begin{proof}
Let $b_0$ be a point over $e$. The fibre at $b_0$ is $(C_0,J_0)$ and $\sigma_\beta(b_0)=D$ has infinite order by Theorem~\ref{thm:nontorsion}, so $\sigma_\beta$ has infinite order. Let $(T,\tau)$ be the trace, with $\tau\colon T_{\mathcal K}\to\Jfam_{\beta,\eta}$ of finite kernel. By Poincar\'e reducibility there is $\pi\colon\Jfam_{\beta,\eta}\to T_{\mathcal K}$ with $\pi\circ\tau=[n]$ for some $n\ge1$. By the N\'eron mapping property both maps extend over a neighbourhood of $b_0$, and the identity $\pi\circ\tau=[n]$ persists. At $b_0$ we obtain $\tau_{b_0}\colon T\to J_0$ with finite kernel, so $J_0$ contains an abelian subvariety isogenous to $T$. If $\dim T=1$, this contradicts Lemma~\ref{lem:simple31}. If $\dim T=2$, then $\tau$ is an isogeny whose kernel is a finite subgroup of the constant abelian variety $T_{\mathcal K}$, hence constant, so $\Jfam_{\beta,\eta}$ is isomorphic to a constant abelian surface; its principal polarization is constant as well, because the N\'eron--Severi group of a constant abelian variety does not change under the extension $\mathcal K/\Qbar$. By Torelli the curves $C_{\beta(b)}$ would then be isomorphic over $\overline{\mathcal K}$, and the moduli map $\psi_2\circ\beta$ would be constant. This is impossible, because $\psi_2$ is finite (Theorem~\ref{thm:quadratic}) and $\beta$ is nonconstant. Hence $T=0$. By the Lang--N\'eron theorem the N\'eron--Tate height over $\mathcal K$ is positive definite modulo torsion when the trace vanishes \cite{Lang}, so $\delta_\beta>0$.
\end{proof}

By \cite{Yang} the number $\delta_\beta$ is rational.

\subsection{The general comparison}
With trivial trace in hand, the specialization theorem identifies the N\'eron--Tate height of the section with a height on the base curve, and the comparison reduces to comparing two classes in $\Pic(B)\otimes\Q$.

\begin{theorem}\label{thm:general}
In the setting above, let $s_\beta=\delta_\beta/(2d_\beta)$. There is a degree-zero class $\nu_\beta\in\Pic^0(B_{\Qbar})\otimes\Q$ such that for all $b\in B^\circ(\Qbar)$
\begin{equation}\label{eq:general}
\begin{aligned}
\hhat_{\Jfam_{\beta,b}}(\sigma_\beta(b))&=s_\beta\,h_{\inv}(\psi_2(\beta(b)))+h_{\nu_\beta}(b)+O_\beta(1)\\
&=s_\beta\,h_{\inv}(\psi_2(\beta(b)))+O_\beta\Bigl(\sqrt{h_{\inv}(\psi_2(\beta(b)))+1}\Bigr).
\end{aligned}
\end{equation}
If $B$ has genus zero, then $\nu_\beta=0$ and the error is $O_\beta(1)$. For every $n\ge1$ the set of $b\in B^\circ(\Qbar)$ with $[k(b):k]\le n$ and $\sigma_\beta(b)$ torsion is finite.
\end{theorem}

\begin{proof}
By Carney's theorem \cite[Theorem~1.1]{Carney}, which extends Tate's theorem for elliptic surfaces \cite{Tate} and refines Silverman's specialization theorem \cite{SilvermanSpec}, there is a nef adelic $\Q$-line bundle $\overline{\mathcal D}_\beta$ on $B$ with ample underlying class $\mathcal D_\beta$ of degree $\delta_\beta$ such that $\hhat_{\Jfam_{\beta,b}}(\sigma_\beta(b))=h_{\overline{\mathcal D}_\beta}(b)$ for $b\in B^\circ(\Qbar)$; here we use Proposition~\ref{prop:trace}. On the other hand $(\psi_2\circ\beta)^*\Hq\simeq\beta^*\OO(120)$, so $h_{\inv}(\psi_2(\beta(b)))=2h_{\beta^*\OO(1)}(b)+O(1)$. Write $\mathcal D_\beta=\frac{\delta_\beta}{d_\beta}\beta^*\OO(1)+\nu_\beta$ with $\nu_\beta$ of degree zero; changing the adelic metric changes heights by $O(1)$, which gives the first equality. Heights attached to algebraically trivial classes satisfy $h_{\nu_\beta}=O(\sqrt{h_{\beta^*\OO(1)}}+1)$ \cite{HindrySilverman,BombieriGubler}, which gives the second. If $B\simeq\PP^1$ then $\Pic^0(B)=0$. Finally, if $\sigma_\beta(b)$ is torsion then $h_{\overline{\mathcal D}_\beta}(b)=0$; since $\mathcal D_\beta$ is ample, such $b$ have bounded height, and Northcott's theorem applies.
\end{proof}

The square-root error in \eqref{eq:general} comes only from the degree-zero class $\nu_\beta$; it can be replaced by $O_\beta(1)$ whenever $\nu_\beta$ is torsion. The slope itself depends only on the curve traced out in the parameter plane.

\begin{proposition}\label{prop:imageonly}
If $\gamma\colon B'\to B$ is a finite morphism of degree $r$, then $\delta_{\beta\circ\gamma}=r\,\delta_\beta$ and $d_{\beta\circ\gamma}=r\,d_\beta$. In particular $s_\beta$ depends only on the curve $\beta(B)\subset\PP^2$.
\end{proposition}

\begin{proof}
Under the finite extension $\Qbar(B')/\Qbar(B)$ of degree $r$, function-field heights normalized by degrees on the base curves are multiplied by $r$, and so is $\deg\beta^*\OO(1)$. Every $\beta$ factors through the normalization of its image.
\end{proof}

\subsection{Lines through the base point}
The simplest curves through $e$ are lines. For them the base is $\PP^1$, there are no degree-zero classes, and the comparison holds with a bounded error. For $(b_1:c_1)\in\PP^1(\Q)$ let $\beta_{b_1,c_1}\colon\PP^1\to\PP^2$, $(s:t)\mapsto(s:b_1t:c_1t)$. These are the lines through $e$, with $d_\beta=1$.

\begin{corollary}\label{cor:lines}
For each line through $e$ there is $\delta>0$ with
\[
\hhat_J(\sigma(s:t))=\delta\,h(s:t)+O(1)=\tfrac\delta2\,h_{\inv}(\psi_2(s:b_1t:c_1t))+O(1)
\]
for $(s:t)\in B^\circ(\Qbar)$. Only finitely many rational parameters $(s:t)$ give a torsion class.
\end{corollary}

\begin{proof}
This is Theorem~\ref{thm:general} with $B=\PP^1$, combined with $h_{\inv}(\psi_2(\beta(t)))=2h(t)+O(1)$.
\end{proof}

\begin{proof}[Proof of Corollary~\ref{cor:E}]
The line $c=0$ with the parameter $u=b/a$ gives $F^{(2)}_{1,u,0}(x,1)=f_0(x)+ux+u^2(x^5+x)$, that is, the curves $C_u$. Corollary~\ref{cor:lines} gives the height formula and the finiteness of $\Sigma$, which also contains the finitely many $u$ with singular fibre. For $u=0$ the class has infinite order by Theorem~\ref{thm:nontorsion}. For $u$ in $\{\pm1,\pm2,\pm3,\pm\frac13,\frac12,\frac23,\frac32\}$ the primitive sextic has nonzero discriminant, $C_u$ has good reduction at the primes listed below, and Lemma~\ref{lem:criteria}(2) applies with the following group orders:
\[
\begin{array}{c|cccccc}
u&1&-1&2&-2&3&-3\\\hline
p&3,\,5&7,\,11&7,\,29&3,\,5&5,\,13&13,\,17\\
\#\Jac(C_u)(\F_p)&16,\,55&65,\,116&42,\,695&16,\,25&25,\,172&174,\,235
\end{array}
\]
\[
\begin{array}{c|ccccc}
u&\frac13&-\frac13&\frac12&\frac23&\frac32\\[2pt]\hline
p&3,\,7&3,\,5&5,\,29&3,\,13&7,\,11\\
\#\Jac(C_u)(\F_p)&12,\,25&12,\,25&25,\,764&12,\,259&25,\,136
\end{array}
\]
In each case $\Jac(C_u)(\Q)$ is torsion-free and $[P_\infty^+-W]\ne0$. Finally, $h_{\inv}([C_u])=2h(u)+O(1)$, so the inequality $\exp h_{\inv}([C_u])\le Y$ holds for every $u$ with $h(u)\le\frac12\log Y-c_1$ and only for $u$ with $h(u)\le\frac12\log Y+c_2$. By Schanuel's theorem \cite{Schanuel} the number of such $u\in\Q$ is $\asymp Y$, and the finite map $u\mapsto\psi_2(1:u:0)$ has fibres of bounded size.
\end{proof}

\subsection{Elliptic curves in \texorpdfstring{$E\times E$}{E x E}}
For curves of positive genus a linear term may appear. The most natural examples are the elliptic curves through the origin of $E\times E$, where the linear term can be written with the N\'eron--Tate pairing on $E$. For $(u,v)\in\Z^2\setminus\{(0,0)\}$ let $\alpha_{u,v}\colon E\to A$, $T\mapsto([u]T,[v]T)$, and $\beta_{u,v}=\phi\circ\alpha_{u,v}$. Then $\beta_{u,v}(O)=e$ and $d_{\beta_{u,v}}=\deg\alpha_{u,v}^*L=3(u^2+v^2)$. Write $U_{u,v}=B^\circ$, $\delta_{u,v}=\delta_{\beta_{u,v}}$ and $\hhat_J(\sigma_{u,v}(T))$ for the height of the specialized section.

\begin{theorem}\label{thm:ray}
For every $(u,v)\ne(0,0)$ there are $\delta_{u,v}>0$ and a unique $R_{u,v}\in E(\Q)\otimes\Q$ such that, uniformly for $T\in U_{u,v}(\Qbar)$,
\begin{align}
\hhat_J(\sigma_{u,v}(T))&=\delta_{u,v}\hhat_E(T)+2\langle T,R_{u,v}\rangle_E+O_{u,v}(1),\label{eq:ray-nt}\\
h_{\inv}(\Phi_2(\alpha_{u,v}(T)))&=6(u^2+v^2)\hhat_E(T)+O_{u,v}(1),\label{eq:ray-git}
\end{align}
and
\begin{multline}\label{eq:ray-ratio}
\hhat_J(\sigma_{u,v}(T))=\frac{\delta_{u,v}}{6(u^2+v^2)}\,h_{\inv}(\Phi_2(\alpha_{u,v}(T)))\\
+O_{u,v}\Bigl(\sqrt{h_{\inv}(\Phi_2(\alpha_{u,v}(T)))+1}\Bigr).
\end{multline}
Moreover $(\Phi_2\circ\alpha_{u,v})^*\Hq\simeq\OO_E(360(u^2+v^2)O)$.
\end{theorem}

\begin{proof}
Since $L$ is symmetric, $[u]^*\OO_E(3O)\simeq\OO_E(3u^2O)$, so $\alpha_{u,v}^*L\simeq\OO_E(3(u^2+v^2)O)$ and $(\Phi_2\circ\alpha_{u,v})^*\Hq\simeq\alpha_{u,v}^*L^{\otimes120}\simeq\OO_E(360(u^2+v^2)O)$; dividing by $60$ gives \eqref{eq:ray-git}. In Theorem~\ref{thm:general} write the class of $\mathcal D_\beta$ in $\Pic(E_{\Qbar})\otimes\Q$ as $\delta_{u,v}[O]+[R-O]$ with $R\in E(\Qbar)\otimes\Q$. We have $h_{[O]}=\hhat_E+O(1)$ and $h_{[R]-[O]}(T)=-2\langle T,R\rangle_E+O(1)$, and we put $R_{u,v}=-R$; this gives \eqref{eq:ray-nt}. If $\delta'\hhat_E+2\langle\cdot,R'\rangle_E$ also satisfies \eqref{eq:ray-nt}, then evaluating at $[n]T_1$ for a point $T_1$ of infinite order shows $\delta'=\delta_{u,v}$ and $\langle T_1,R'-R_{u,v}\rangle_E=0$ for all such $T_1$, so $R'=R_{u,v}$ since the pairing is positive definite on $E(F)\otimes\R$ for every number field $F$. The family is defined over $\Q$, so the left side of \eqref{eq:ray-nt} is Galois invariant, and uniqueness gives $R_{u,v}^\tau=R_{u,v}$ for all $\tau\in\Gal(\Qbar/\Q)$; hence $R_{u,v}\in E(\Q)\otimes\Q$. Finally $|2\langle T,R_{u,v}\rangle_E|\le2\sqrt{\hhat_E(T)\hhat_E(R_{u,v})}$, and \eqref{eq:ray-git} turns this into \eqref{eq:ray-ratio}.
\end{proof}

\begin{proposition}\label{prop:symmetries}
For $(u,v)\ne(0,0)$ and $d\in\Z\setminus\{0\}$,
\[
\delta_{du,dv}=d^2\delta_{u,v},\quad R_{du,dv}=d\,R_{u,v},\quad\delta_{v,u}=\delta_{u,v},\quad R_{v,u}=R_{u,v}.
\]
Hence the slope $s(u:v)=\delta_{u,v}/(6(u^2+v^2))$ is a well-defined function on $\PP^1(\Q)$ with $s(v:u)=s(u:v)$.
\end{proposition}

\begin{proof}
Since $\alpha_{du,dv}=\alpha_{u,v}\circ[d]$, the fibre over $T$ for $(du,dv)$ is the fibre over $[d]T$ for $(u,v)$, and $U_{du,dv}=[d]^{-1}U_{u,v}$. Thus $\hhat_J(\sigma_{du,dv}(T))=\delta_{u,v}d^2\hhat_E(T)+2\langle T,dR_{u,v}\rangle_E+O(1)$, and uniqueness in Theorem~\ref{thm:ray} gives the first two identities. Since $\phi\circ s=\phi$ and $\alpha_{v,u}=s\circ\alpha_{u,v}$, the families for $(u,v)$ and $(v,u)$ coincide.
\end{proof}

\begin{theorem}\label{thm:canonical-identities}
Let $(u,v)\ne(0,0)$ and let $T\in E(\Qbar)$ have infinite order. Then $[n]T\in U_{u,v}$ for all but finitely many $n$, and
\begin{align*}
\lim_{|n|\to\infty}\frac{\hhat_J(\sigma_{u,v}([n]T))}{n^2}&=\delta_{u,v}\,\hhat_E(T),\\
\lim_{|n|\to\infty}\frac{h_{\inv}(\Phi_2(\alpha_{u,v}([n]T)))}{n^2}&=\hhat_{\Phi_2\circ\alpha_{u,v},\ww}(T)=6(u^2+v^2)\,\hhat_E(T),\\
\lim_{n\to\infty}\frac{\hhat_J(\sigma_{u,v}([n]T))-\hhat_J(\sigma_{u,v}([-n]T))}{4n}&=\langle T,R_{u,v}\rangle_E.
\end{align*}
In particular
\[
\lim_{|n|\to\infty}\frac{\hhat_J(\sigma_{u,v}([n]T))}{n^2}=s(u:v)\cdot\hhat_{\Phi_2\circ\alpha_{u,v},\ww}(T),
\]
so the slope is the ratio of the N\'eron--Tate height of the section to the weighted N\'eron--Tate height of the moduli map along the curve $\alpha_{u,v}(E)$.
\end{theorem}

\begin{proof}
The complement of $U_{u,v}$ is finite and $T$ has infinite order, so $[n]T$ avoids it for $|n|$ large. Substitute $[n]T$ in \eqref{eq:ray-nt} and \eqref{eq:ray-git}, and use Theorem~\ref{thm:wnt} with $m=60$ for the weighted N\'eron--Tate height of $\Phi_2\circ\alpha_{u,v}$.
\end{proof}

Along the congruence lattice of Theorem~\ref{thm:grid} the canonical identities become explicit expansions in $n$, valid for every $n$ and not only asymptotically.

\begin{corollary}\label{cor:ray-asymptotic}
Let $M=342$, $T_n=[Mn]P_0$, $\eta(n)=\hhat_J(\sigma_{u,v}(T_n))$ and $\mu(n)=h_{\inv}(\Phi_2(\alpha_{u,v}(T_n)))$. Every fibre is smooth and $\eta(n)>0$ for all $n\in\Z$. With $\kappa_{u,v}=2M\langle P_0,R_{u,v}\rangle_E$,
\begin{align*}
\eta(n)&=\delta_{u,v}M^2\hhat_E(P_0)n^2+\kappa_{u,v}n+O(1),\\
\mu(n)&=6(u^2+v^2)M^2\hhat_E(P_0)n^2+O(1),
\end{align*}
so that $\frac12(\eta(n)+\eta(-n))=s(u:v)\cdot\frac12(\mu(n)+\mu(-n))+O(1)$ and
\[
\#\{n\in\Z:\ \eta(n)\le X\}=\frac{2}{M\sqrt{\delta_{u,v}\hhat_E(P_0)}}\sqrt X+O(1).
\]
\end{corollary}

\begin{proof}
The point $([u]T_n,[v]T_n)=R_{un,vn}$ lies on the lattice of Theorem~\ref{thm:grid}, which gives smoothness and positivity. The expansions follow from Theorem~\ref{thm:ray} and bilinearity. Averaging over $\pm n$ removes the linear term, and completing the square in $\eta(n)\le X$ gives the count.
\end{proof}

\subsection{Comparison with GIT heights}
Finally, Theorem~\ref{thm:B} lets us replace the weighted moduli height by the GIT height in all of these statements, at the cost of an explicit constant in one direction.

\begin{corollary}\label{cor:git-ray}
For every curve $\beta$ as in Theorem~\ref{thm:general} and every $b\in B^\circ(\Qbar)$,
\[
\hhat_{\Jfam_{\beta,b}}(\sigma_\beta(b))\le s_\beta\,h_{\GIT}\bigl(F^{(2)}_{\beta(b)}\bigr)+O_\beta\Bigl(\sqrt{h_{\GIT}\bigl(F^{(2)}_{\beta(b)}\bigr)+1}\Bigr),
\]
with error $O_\beta(1)$ when $B$ has genus zero. With Remark~\ref{rem:reverse}, the N\'eron--Tate height is in fact $s_\beta h_{\GIT}+O(\sqrt{h_{\GIT}+1})$.
\end{corollary}

\begin{proof}
Combine \eqref{eq:general} with $h_{\inv}\le h_{\GIT}+\frac12\log302$ (Theorem~\ref{thm:B}).
\end{proof}

\section{Further questions}\label{sec:questions}

We close with some questions that the methods of this paper leave open.

\subsection*{The value of the slope}
For the line $c=0$ the family $C_u$ is a genus-two fibration over $\PP^1$ whose coefficients have degree at most two in $u$. Its slope $\delta/2$ can in principle be computed from the intersection theory of the minimal regular model, together with local corrections at the singular fibres, in the spirit of \cite{MS}. A numerical value would make Corollary~\ref{cor:E} fully explicit.

\subsection*{Variation of the slope}
Proposition~\ref{prop:imageonly} defines a function $\mathcal C\mapsto s(\mathcal C)$ on irreducible curves $\mathcal C\subset\PP^2$ through $e$. It would be interesting to know whether $s$ is bounded above and below by positive constants, and how it behaves on the lines through $e$ and on the curves $\phi(\alpha_{u,v}(E))$. An upper bound is related to uniform comparisons between canonical and naive heights in families \cite{CallSilverman}, and a lower bound to the height inequality of Gao and Habegger \cite{GaoHabegger} and of Dimitrov, Gao and Habegger \cite{DGH}, applied to the image of $\sigma$ over the two-dimensional base.

\subsection*{Two-dimensional specialization}
Theorem~\ref{thm:grid} gives an explicit infinite set of non-torsion specializations over the rank-two lattice $\Z P_0\times\Z P_0$, while Masser's results \cite{Masser} control exceptional specializations of finitely generated subgroups in density. A comparison $\hhat_J(\sigma(R))\asymp h_{\inv}(\Phi_2(R))$ on a dense open subset of the two-dimensional base would combine both.

\subsection*{Archimedean orbit minima}
Theorem~\ref{thm:B} bounds $h_{\inv}-h_{\GIT}$ from above. The opposite inequality depends on the archimedean minima $\rho_\infty(F)$, and explicit lower bounds for $\rho_\infty$ in terms of the invariants would give a fully effective comparison between weighted and GIT heights of binary sextics.

\appendix
\section{Computational certificates}\label{app:data}

This appendix lists the data behind the computational steps of the proofs. All computations are exact and were carried out with SymPy \cite{SymPy} and FLINT \cite{FLINT}; they take a few seconds, can be reproduced in any computer algebra system, and are collected in a script provided as supplementary material.

\subsection*{Invariants}
The transvectant formulas \eqref{eq:IC} applied to $f_0$ give
\begin{gather*}
I_2=794166=2\cdot3\cdot132361,\qquad I_4=19286078400,\\
I_6=4480732871164800,\qquad I_{10}=\Disc(f_0)=4130124925538304000000.
\end{gather*}
The same values are obtained from Igusa's root formulas, which confirms the normalization. For the linear family the discriminant is
\begin{equation}\label{eq:I10lin}
\begin{aligned}
-I_{10}={}&46656a^{10}-38880a^8b^2+38880a^8bc-202176a^8c^2\\
&+64800a^7b^2c-375840a^7bc^2+215568a^7c^3-540a^6b^4\\
&+109080a^6b^3c-343116a^6b^2c^2+529200a^6bc^3-48816a^6c^4\\
&+48600a^5b^4c-222480a^5b^3c^2+605376a^5b^2c^3-103176a^5bc^4\\
&+25668a^5c^5-6660a^4b^6+28980a^4b^5c-175221a^4b^4c^2\\
&+423270a^4b^3c^3-100611a^4b^2c^4-14238a^4bc^5-28081a^4c^6\\
&+8400a^3b^6c-89520a^3b^5c^2+185444a^3b^4c^3-84128a^3b^3c^4\\
&-24524a^3b^2c^5-10640a^3bc^6+7092a^3c^7-320a^2b^8\\
&+9280a^2b^7c-43256a^2b^6c^2+59960a^2b^5c^3-24450a^2b^4c^4\\
&-15372a^2b^3c^5+6568a^2b^2c^6+6790a^2bc^7-1536a^2c^8\\
&+3200ab^8c-12160ab^7c^2+12112ab^6c^3-496ab^5c^4\\
&-4452ab^4c^5+960ab^3c^6+1400ab^2c^7-596abc^8\\
&-92ac^9-256b^{10}+1280b^9c-1856b^8c^2\\
&+320b^7c^3+1052b^6c^4-332b^5c^5-357b^4c^6\\
&+126b^3c^7+23b^2c^8.
\end{aligned}
\end{equation}
For the quadratic family the invariants $I_2,I_4,I_6,I_{10}$ are forms of degrees $4,8,12,20$ in $(a,b,c)$ with $12,40,88,219$ monomials; for example
\begin{multline*}
I_2=2(397083a^4+380a^3b+312a^3c-11620a^2b^2-7288a^2bc\\
-7220a^2c^2+20ab^3-8abc^2-8ac^3+20b^4+12b^2c^2-8bc^3).
\end{multline*}

\subsection*{Semistability and bad primes}
The resultant of the two jet minors in Proposition~\ref{prop:linear} is $2^{45}3^{22}5^{16}\cdot761\cdot14447$; since only two of the four minors are used, this bound is weaker than Proposition~\ref{prop:linear-git}. For that proposition,
\begin{align*}
N_1&=2^{368}3^{58}5^{117}\cdot19\cdot821\cdot1031\cdot14447^2\cdot76103\cdot43231361\cdot P_{23},\\
N_2&=2^{536}3^{70}5^{164}\cdot1493\cdot14447^4\cdot381977\cdot7176343\cdot28907983207\cdot P_{30}\cdot P_{33},
\end{align*}
where $P_{23}=81164371046060827598167$ and $P_{30}$, $P_{33}$ are primes with $30$ and $33$ digits. For the quadratic family, the four polynomials $F,\partial_xF,\partial_x^2F,\partial_x^3F$ have reduced Gr\"obner basis $\{1\}$ over $\Q$ in each chart $a=1$, $b=1$, $c=1$, for the graded reverse lexicographic order.

\subsection*{Point counts}
The counts for $C_0$ used in the paper are
\[
\begin{array}{c|ccccc}
p&\#C_0(\F_p)&\#C_0(\F_{p^2})&P_p(T)&\#J_0(\F_p)\\\hline
13&14&182&T^4+6T^2+169&176\\
17&18&286&T^4-2T^2+289&288\\
31&36&942&T^4+4T^3-2T^2+124T+961&1088
\end{array}
\]
The curve $E$ satisfies $\#E(\F_{13})=19$ and $\#E(\F_{17})=18$. The point counts for Proposition~\ref{prop:small} and Corollary~\ref{cor:E} use the formula $\#J(\F_p)=P_p(1)$ of Section~\ref{sec:sections}; when the leading coefficient of the sextic is a square, the two points at infinity are rational.

\subsection*{Acknowledgment}
This work was supported by the Funda\c{c}\~ao Carlos Chagas Filho de Amparo \`a Pesquisa do Estado do Rio de Janeiro (FAPERJ), grant E-26/210.728/2024, within the project \emph{Rela\c{c}\~ao entre a altura ponderada, altura GIT e altura de N\'eron--Tate} at the Universidade do Estado do Rio de Janeiro.

\end{document}